\documentclass{amsart}
\usepackage{amssymb, amsfonts, amsthm}
\usepackage{graphics, epstopdf, pstricks, psfrag}

\usepackage{tikz}
\usetikzlibrary{arrows,decorations.pathmorphing,decorations.markings,backgrounds,positioning,fit,calc,shapes}
\tikzstyle{mutable}=[inner sep=0.5mm,circle,draw,minimum size=2mm]
\tikzstyle{frozen}=[inner sep=.9mm,rectangle,draw]
\tikzstyle{dot} = [fill=black!25,inner sep=0.5mm,circle,draw,minimum size=1mm]
\tikzstyle{blue dot} = [draw=blue,fill=blue!25,inner sep=0.5mm,circle,draw,minimum size=1mm]
\tikzstyle{marked}=[inner sep=0.5mm,circle,draw,blue!75!black,fill=blue!50]
\tikzstyle{outline}=[thick,line width=1.5mm,draw=black!10]
\tikzstyle{oriented}=[draw=black,thick,decoration={markings,mark=at position 0.52 with {\arrow{>}}},postaction={decorate}]
\tikzstyle{faded oriented}=[draw=black!25,thick,decoration={markings,mark=at position 0.52 with {\arrow{>}}},postaction={decorate}]\tikzstyle{invisible}=[inner sep=-.3, minimum size=-.3]

\usepackage{xkeyval}
\newcommand{\PostnikovBlack}[5][]{
    \pgfmathanglebetweenpoints{\pgfpointanchor{#2}{center}}
                              {\pgfpointanchor{#3}{center}}
    \global\let\firstangle\pgfmathresult 
    \pgfmathanglebetweenpoints{\pgfpointanchor{#2}{center}}
                              {\pgfpointanchor{#4}{center}}
    \global\let\secondangle\pgfmathresult; 
    \pgfmathanglebetweenpoints{\pgfpointanchor{#2}{center}}
                              {\pgfpointanchor{#5}{center}}
    \global\let\thirdangle\pgfmathresult; 
  		\path (#2) to coordinate[midway](m1) (#3);
		\path (#2) to coordinate[midway](m2) (#4);
		\path (#2) to coordinate[midway](m3) (#5);
		\draw[oriented,#1] (m1) to [out=\firstangle-135,in=\secondangle+135] (m2);
		\draw[oriented,#1] (m2) to [out=\secondangle-135,in=\thirdangle+135] (m3);
		\draw[oriented,#1] (m3) to [out=\thirdangle-135,in=\firstangle+135] (m1);
}
\newcommand{\PostnikovWhite}[5][]{
    \pgfmathanglebetweenpoints{\pgfpointanchor{#2}{center}}
                              {\pgfpointanchor{#3}{center}}
    \global\let\firstangle\pgfmathresult 
    \pgfmathanglebetweenpoints{\pgfpointanchor{#2}{center}}
                              {\pgfpointanchor{#5}{center}}
    \global\let\secondangle\pgfmathresult; 
    \pgfmathanglebetweenpoints{\pgfpointanchor{#2}{center}}
                              {\pgfpointanchor{#4}{center}}
    \global\let\thirdangle\pgfmathresult; 
  		\path (#2) to coordinate[midway](m1) (#3);
		\path (#2) to coordinate[midway](m2) (#5);
		\path (#2) to coordinate[midway](m3) (#4);
		\draw[oriented,#1] (m1) to [out=\firstangle+135,in=\secondangle-135] (m2);
		\draw[oriented,#1] (m2) to [out=\secondangle+135,in=\thirdangle-135] (m3);
		\draw[oriented,#1] (m3) to [out=\thirdangle+135,in=\firstangle-135] (m1);
}
\newcommand{\PostnikovWhiteBoundaryThree}[5][]{
    \pgfmathanglebetweenpoints{\pgfpointanchor{#2}{center}}
                              {\pgfpointanchor{#3}{center}}
    \global\let\firstangle\pgfmathresult 
    \pgfmathanglebetweenpoints{\pgfpointanchor{#2}{center}}
                              {\pgfpointanchor{#5}{center}}
    \global\let\secondangle\pgfmathresult; 
    \pgfmathanglebetweenpoints{\pgfpointanchor{#2}{center}}
                              {\pgfpointanchor{#4}{center}}
    \global\let\thirdangle\pgfmathresult; 
    		\path (#2) to coordinate[midway](m2) (#5);
		\path (#2) to coordinate[midway](m3) (#4);
		\draw[oriented,#1] (#3) to [out=\firstangle+135,in=\secondangle-135] (m2);
		\draw[oriented,#1] (m2) to [out=\secondangle+135,in=\thirdangle-135] (m3);
		\draw[oriented,#1] (m3) to [out=\thirdangle+135,in=\firstangle-135] (#3);
}
\newcommand{\PostnikovBlackBoundaryThree}[5][]{
    \pgfmathanglebetweenpoints{\pgfpointanchor{#2}{center}}
                              {\pgfpointanchor{#3}{center}}
    \global\let\firstangle\pgfmathresult 
    \pgfmathanglebetweenpoints{\pgfpointanchor{#2}{center}}
                              {\pgfpointanchor{#4}{center}}
    \global\let\secondangle\pgfmathresult; 
    \pgfmathanglebetweenpoints{\pgfpointanchor{#2}{center}}
                              {\pgfpointanchor{#5}{center}}
    \global\let\thirdangle\pgfmathresult; 
		\path (#2) to coordinate[midway](m2) (#4);
		\path (#2) to coordinate[midway](m3) (#5);
		\draw[oriented,#1] (#3) to [out=\firstangle-135,in=\secondangle+135] (m2);
		\draw[oriented,#1] (m2) to [out=\secondangle-135,in=\thirdangle+135] (m3);
		\draw[oriented,#1] (m3) to [out=\thirdangle-135,in=\firstangle+135] (#3);
}
\newcommand{\PostnikovWhiteBIII}[6][]{
    \pgfmathanglebetweenpoints{\pgfpointanchor{#2}{center}}
                              {\pgfpointanchor{#3}{center}}
    \global\let\firstangle\pgfmathresult 
    \pgfmathanglebetweenpoints{\pgfpointanchor{#2}{center}}
                              {\pgfpointanchor{#6}{center}}
    \global\let\secondangle\pgfmathresult; 
    \pgfmathanglebetweenpoints{\pgfpointanchor{#2}{center}}
                              {\pgfpointanchor{#5}{center}}
    \global\let\thirdangle\pgfmathresult; 
    \pgfmathanglebetweenpoints{\pgfpointanchor{#2}{center}}
                              {\pgfpointanchor{#4}{center}}
    \global\let\fourthangle\pgfmathresult; 
    		\path (#2) to coordinate[midway](m2) (#6);
		\path (#2) to coordinate[midway](m3) (#5);
		\path (#2) to coordinate[midway](m4) (#4);
		\draw[oriented,#1] (#3) to [out=\firstangle+135,in=\secondangle-135] (m2);
		\draw[oriented,#1] (m2) to [out=\secondangle+135,in=\thirdangle-135] (m3);
		\draw[oriented,#1] (m3) to [out=\thirdangle+135,in=\fourthangle-135] (m4);
		\draw[oriented,#1] (m4) to [out=\fourthangle+135,in=\firstangle-135] (#3);
}
\newcommand{\PostnikovBlackBBII}[6][]{
    \pgfmathanglebetweenpoints{\pgfpointanchor{#2}{center}}
                              {\pgfpointanchor{#3}{center}}
    \global\let\firstangle\pgfmathresult 
    \pgfmathanglebetweenpoints{\pgfpointanchor{#2}{center}}
                              {\pgfpointanchor{#4}{center}}
    \global\let\secondangle\pgfmathresult; 
    \pgfmathanglebetweenpoints{\pgfpointanchor{#2}{center}}
                              {\pgfpointanchor{#5}{center}}
    \global\let\thirdangle\pgfmathresult; 
    \pgfmathanglebetweenpoints{\pgfpointanchor{#2}{center}}
                              {\pgfpointanchor{#6}{center}}
    \global\let\fourthangle\pgfmathresult; 
		\path (#2) to coordinate[midway](m2) (#5);
		\path (#2) to coordinate[midway](m3) (#6);
		\draw[oriented,#1] (#3) to [out=\firstangle-165,in=\secondangle+165] (#4);
		\draw[oriented,#1] (#4) to [out=\secondangle-135,in=\thirdangle+135] (m2);
		\draw[oriented,#1] (m2) to [out=\thirdangle-135,in=\fourthangle+135] (m3);
		\draw[oriented,#1] (m3) to [out=\fourthangle-135,in=\firstangle+135] (#3);
}

\usepackage{multirow}

\title{Cluster Algebras of Grassmannians are Locally Acyclic}

\author{Greg Muller, David E Speyer}

\address{Department of Mathematics, University of Michigan, 
         Ann Arbor, MI 48109}

\keywords{}

\thanks{}

\theoremstyle{plain}
\newtheorem{theorem}{Theorem}[section]
\newtheorem{conjecture}[theorem]{Conjecture}
\newtheorem{proposition}[theorem]{Proposition}

\newtheorem{lemma}[theorem]{Lemma}
\newtheorem{corollary}[theorem]{Corollary}

\theoremstyle{definition}
\newtheorem{definition}[theorem]{Definition}

\theoremstyle{remark}
\newtheorem{remark}[theorem]{Remark}
\newtheorem{example}[theorem]{Example}

\newcommand{\ZZ}{\mathbb{Z}}
\newcommand{\QQ}{\mathbb{Q}}

\newcommand{\cA}{\mathcal{A}}

\newcommand{\Q}{\mathsf{Q}}
\newcommand{\tQ}{\widetilde{\mathsf{Q}}}

\newcommand{\newword}[1]{\textbf{\emph{#1}}}

\newcommand{\A}{\mathcal{A}}

\DeclareMathOperator{\Spec}{Spec}

\def\Pio{\mathring{\Pi}}

\begin{document}

\begin{abstract}
Considered as commutative algebras, cluster algebras can be very unpleasant objects.
However, the first author introduced a condition known as ``local acyclicity" which implies that cluster algebras behave reasonably.
One of the earliest and most fundamental examples of a cluster algebra is the homogenous coordinate ring of the Grassmannian.
We show that the Grassmannian is locally acyclic. Morally, show the stronger result that all positroid varieties are locally acyclic. However, it has not been shown that all positroid varieties have cluster structure in the expected manner, so what we actually prove is that certain cluster varieties associated to Postnikov's alternating strand diagrams are locally acylic.
We actually establish a slightly stronger property than local acyclicity, that is designed to facilitate proofs involving the Mayer-Vietores sequence.
\end{abstract}

\maketitle

\section{Introduction}

A cluster algebra is a commutative algebra $A$ over a field $k$, equipped with a collection of subsets of $A$ known as \newword{clusters}. (In this paragraph we are describing cluster algebras, not defining them; definitions can be found in Section~\ref{ClusterBackground}. Also, there are many slightly different definitions in the literature.)
Each cluster $(x_1, x_2, \ldots, x_r)$ of $A$ has the property that $A$ is contained in the Laurent polynomial ring $k[x_1^{\pm 1}, x_2^{\pm 1}, \ldots, x_r^{\pm 1}]$ and there are combinatorial rules describing how to take one cluster and obtain another. 
The elements of the clusters are known as \newword{cluster variables}.
Geometrically, one should think of each cluster as describing an open torus inside $\Spec A$. 

What if one inverts not all the variables in a cluster, but just one cluster variable $x$? $\Spec A[x^{-1}]$ is an open subvariety of $\Spec A$. 
In nice cases, $A[x^{-1}]$ is itself a cluster algebra, whose clusters are those clusters of $A$ that contain $x$.
What would be particularly nice would be if we had two cluster variables $x_1$ and $x_2$, which had no common zeroes in $\Spec A$, so that $A[x_1^{-1}]$ and $A[x_2^{-1}]$ were both cluster algebras, as then $\Spec A$ would be the union of two open subvarieties, $\Spec A[x_1^{-1}]$ and $\Spec A[x_2]^{-1}$, which are simpler cluster algebras.
Even nicer than that would be if $A[x_1^{-1}, x_2^{-1}]$ were also a cluster algebra, so that we would have a Mayer-Vietores decomposition $\Spec A = \Spec A[x_1^{-1}] \cup \Spec A[x_2^{-1}]$, $\Spec A[x_1^{-1}, x_2^{-1}] = \Spec A[x_1^{-1}] \cap \Spec A[x_2^{-1}]$ all of whose terms were cluster algebras.
In an ideal world, each of these simpler cluster algebras would, in turn, have such a Mayer-Vietores decomposition and so on.

In Section~\ref{ClusterBackground}, we will define what it means for a cluster algebra to be \newword{Louise}.\footnote{The first author described the \newword{Banff algorithm}, named for the city in which it was found. The Banff algorithm searches for cluster variables $x_1$ and $x_2$ such that $\Spec A = \Spec A[x_1^{-1}] \cup \Spec A[x_2^{-1}]$; we will define a cluster algebra to be \newword{Banff} when the Banff algorithm succeeds. Lake Louise is the prettiest part of the city of Banff, and we thus define  a cluster algebra to be \newword{Louise} if it is Banff in the best possible way, permitting Mayer-Vietores arguments.} 
In a Louise cluster algebra, the ideal situation of the preceding paragraph holds.

One of the earliest motivating examples of a cluster algebra was the homogenous coordinate ring of the Grassmannian in its Pl\"ucker embedding \cite{Sco06}.
A major result of this paper is that this cluster algebra is Louise.

Because the definition of cluster algebra used in this (and many other papers) has changed in the last decade, from our perspective, the cluster algebra in question is actually the coordinate ring of the open subvariety of the Grassmannian where the $n$ Pl\"ucker coordinates $p_{123 \cdots k}$, $p_{234 \cdots k(k+1)}$, \dots, $p_{(n-k+1) \cdots n}$, $p_{(n-k+2) \cdots n1}$, \dots, $p_{n12\cdots (k-1)}$ are nonzero.
This open locus is the largest \newword{positroid cell}. 
The Grassmannian is stratified into numerous affine subvarieties known as \newword{positroid cells}.

It is expected that the homogeneous coordinate ring of every positroid variety is a cluster algebra, and that the structure of these cluster algebras will be described by combinatorial objects known as Postnikov diagrams (also called alternating strand diagrams). 
There is, unambiguously, a cluster algebra associated to any Postnikov diagram; the open question is whether this cluster algebra is the homogeneous coordinate ring of the open positroid variety.
Leclerc~\cite{Lec14} has recently constructed a cluster structure in the coordinate ring of any positroid variety; it is not yet clear what the precise relation is between Leclerc's construction and Postnikov diagrams.
In this paper, we will avoid this issue, but we will show that the cluster algebra associated to any Postnikov diagram is Louise.
Thus, once it is proved that these cluster algebras are the coordinate rings of open positroid varieties, we will know that every open positroid variety is Louise.
For the largest open positroid variety, described in the previous paragraph, Scott's result~\cite{Sco06} establishes that the cluster algebra is the coordinate ring of the open positroid variety. So we have shown that this open locus in the Grassmannian is a locally acyclic cluster algebra, as promised in our title.

We also remark that double Bruhat cells for $GL_n$~\cite{BFZ05} can be realized as positroid varieties~\cite[Section 6]{KLS13}, \cite{BGY06}, with quivers coming from Postnikov diagrams.
Hence our result also shows that double Bruhat cells are locally acyclic.

In Section~\ref{ClusterBackground}, we review background on cluster algebras, local acyclicity and the Louise condition.
In Section~\ref{PositroidBackground}, we review background on positroid varieties and Postnikov diagrams.
In Section~\ref{Main}, we prove the main theorem.
We note that the original results of this paper are almost entirely combinatorial; the geometry which we have discussed in this introduction is mostly outsourced to earlier papers. 
The reader who hates or fears algebraic geometry should still be able to understand our proof, although not its motivation.

\section{Cluster algebras and local acyclicity} \label{ClusterBackground}

%

The specific form of cluster algebras we use in this note is that of \emph{skew-symmetric} cluster algebras of \emph{geometric type}, and we define them using \emph{ice quivers}; a more thorough introduction to this approach can be found in \cite{Kel12}.

A cluster algebra determines, and is determined by, a \emph{mutation-equivalence class of seeds} in some field $\mathcal{F}$.  A \newword{seed} in $\mathcal{F}$ consists of the following data.
\begin{itemize}
	\item A \newword{quiver} $\tQ$ without loops or directed 2-cycles.
	\item A subset of the vertices of $\tQ$ designated as \newword{frozen}; the rest are called \newword{mutable}.
	\item A bijection from the vertices of $\tQ$ to a free generating set $\mathbf{x}$ of the field $\mathcal{F}$ over $\QQ$.  The image $x_i$ of a vertex $i$ is called the \newword{cluster variable} at that vertex, and the set $\mathbf{x}=\{x_1,x_2,...,x_n\}$ is called a \newword{cluster}.
\end{itemize}
\noindent The data of the quiver $\tQ$ and the designation of vertices as frozen and mutable is known as an \newword{ice quiver}.
The \newword{mutable quiver} $\Q$ of a seed is the induced subquiver on the mutable vertices.

With the above notation, if $i$ is a mutable vertex of $\tQ$, we'll write $\tQ[i^{-1}]$ for the same quiver $\tQ$ with the vertex $i$ designated as frozen. 
The mutable part of $\tQ[i^{-1}]$ is thus the quiver obtained by deleting vertex $i$ from $\Q$; we will denote this by $\Q[i^{-1}]$.\footnote{The notation $\Q[i^{-1}]$ is ambiguous if $\Q$ is a directed graph encountered in the abstract, without either being treated as an ice quiver or as being the mutable part of a quiver, but this will never cause ambiguity for us.}
See Theorem~\ref{Banff} for the motivation for this notation.

We need the following graph theoretic definition: If $\Q$ is a quiver, then a \newword{bi-infinite directed path} in $\Q$ is a sequence of arrows indexed by $\ZZ$, such that the target of the $i$-th arrow is the source of the $(i+1)$-st arrow.

A seed may be \newword{mutated} at any mutable vertex, yielding a new seed in the same field, and iterated mutation generates an equivalence relation among seeds (see \cite{Kel12} for details).  Given a mutation-equivalence class of seeds in $\mathcal{F}$, the corresponding \newword{cluster algebra} $\cA$ is the subring of the field $\mathcal{F}$ generated by the cluster variables in those seeds and the inverses to the frozen variables.  Note that, up to canonical isomorphism, the cluster algebra $\cA$ only depends on the ice quiver of a single seed.

The \newword{exchange type} of a cluster algebra is its set of mutable quivers, which may be regarded as an equivalence class of quivers, up to the mutation.
Many properties of a cluster algebra are known to depend only on its exchange type, including the \emph{Louise} property considered in this note.  

In \cite{MulLA}, the first author introduced the class of locally acyclic cluster algebras.  A cluster algebra $\cA$ is \newword{locally acyclic} if there are localizations $\cA[z_1^{-1}]$, $\cA[z_2^{-1}]$, \dots, and $\cA[z_M^{-1}]$, where each $z_M$ is a product of cluster variables within a single cluster, such that 
\begin{enumerate}
	\item Each $\cA[z_i^{-1}]$ is a cluster algebra with a seed whose mutable quiver has no directed cycle (these are called \emph{acyclic cluster algebras}).
	\item $\text{Spec}(\cA) = \bigcup_{i=1}^n \text{Spec}(\cA[z_i^{-1}]) $.
\end{enumerate}
Known local algebro-geometric properties of these elementary cluster algebras can then be extended to cluster algebras in this larger class.
\begin{theorem}\label{thm:LACA} \cite{MulLA}
If $\cA$ is a locally acyclic cluster algebra, then $\cA$ is finitely generated, normal, locally a complete intersection, and equal to its own upper cluster algebra (the relevant definitions may be found in \emph{loc. cit.}).
\end{theorem} 
\noindent Recent work \cite{BMRS14} has also shown that locally acyclic cluster algebras have (at worst) canonical singularities.

In \cite{MulLA}, the principal tool for demonstrating local acyclicity is the \emph{Banff algorithm}.  This is based on the following observation:
\begin{theorem} \label{Banff} \cite{MulLA}
Let $\cA$ be a cluster algebra and $\tQ$ the quiver of some seed. Suppose that the mutable quiver $\Q$ contains two vertices $s$ and $t$ connected by an arrow which is not in any bi-infinite directed path of $\Q$. Then $\Spec(\cA[x_s^{-1}])$ and $\Spec(\cA[x_t^{-1}])$ cover $\Spec(\cA)$.  
Furthermore, if the cluster algebras corresponding to the ice quivers $\tQ[s^{-1}]$ and $\tQ[t^{-1}]$ are locally acyclic, then they are equal to $\cA[x_s^{-1}]$ and $\cA[x_t^{-1}]$ respectively.
\end{theorem}

Of course, it may not be clear whether the localizations are locally acyclic, so the natural idea is to iterate the previous argument, refining the cover.  If this \newword{Banff algorithm} can eventually produce a cover by acyclic cluster algebras, we say the cluster algebra is \newword{Banff}.
The Banff property only depends on the exchange type of a cluster algebra, and can be checked using only quivers.\footnote{It is unknown to the authors whether the local acyclicity property only depends on the exchange type.}
 
We have a nicer condition in mind, however.  For future applications based on Mayer-Vietoris sequences, we will need to consider not just open covers, but some intersections between open sets in those covers.  To that end, we define the notion of a Louise quiver. 

\begin{definition}\label{defn: Louise}
The class of \newword{Louise} quivers is the smallest class of quivers satisfying the following conditions.
\begin{enumerate}
         \item Any edgeless quiver is Louise.
	\item The mutation of any Louise quiver is Louise; that is, the class is closed under mutation.
	\item Let $\Q$ be a quiver with an arrow $\alpha$ from $s$ to $t$, such that $\alpha$ is not contained in any bi-infinite directed path.  If $\Q[t^{-1}]$, $\Q[s^{-1}]$ and $\Q[s^{-1},t^{-1}]$ are all Louise, then $\Q$ is Louise.
\end{enumerate}
A cluster algebra is \newword{Louise} if its exchange type consists of Louise quivers.
\end{definition}
\begin{remark}
If part $(3)$ of Definition \ref{defn: Louise} only checked $\Q[s^{-1}]$ and $\Q[t^{-1}]$, it would instead define the class of Banff quivers and cluster algebras.
\end{remark}
\begin{remark}
In practice, we will only ever use a simpler but less general version of Condition (3), where the vertex $s$ is a source.  This immediately implies that $\alpha$ is not in any bi-infinite path.  This less general condition also has stronger consequences; see Remark \ref{rem: nondeg}.
\end{remark}

Like algebras obeying the Banff condition, Louise cluster algebras are locally acyclic.

\begin{proposition}
A Louise cluster algebra is locally acyclic.
\end{proposition}
\begin{proof}
The proof is by induction on the number of vertices $m$ in any mutable quiver of $\cA$.  If $m=0$, then $\Q$ is empty and so it is automatically acyclic and locally acyclic.  Next, assume for induction that every Louise cluster algebra with fewer than $m$ vertices in any mutable quiver is locally acyclic. 
If $\Q$ has no arrows, then it is acyclic and automatically locally acyclic.

If $\Q$ is Louise and has any arrows, then there must be a mutation-equivalent quiver $\Q'$ for which part $(3)$ of Definition \ref{defn: Louise} applies; let $s$, $t$ and $\alpha$ be as above.  It follows that $\cA$ has a seed with mutable quiver $\Q'$.   By \cite[Corollary 5.4]{MulLA}, $\Spec(\cA[x_s^{-1}])$ and $\Spec(\cA[x_t^{-1}])$ cover $\Spec(\cA)$, and they have mutable quivers $\Q[s^{-1}]$ and $\Q[t^{-1}]$.  Since the quivers $\Q[s^{-1}]$ and $\Q[t^{-1}]$ are Louise with fewer vertices than $\Q$, any cluster algebra with that exchange type is locally acyclic.  Hence, by \cite[Lemma 3.4]{MulLA}, $\cA[x_s^{-1}]$ and $\cA[x_t^{-1}]$ are locally acyclic cluster algebras.  Then $\Spec(\cA)$ has a cover by the spectra of locally acyclic cluster algebras, so $\A$ is locally acyclic.
\end{proof}

The Louise property was first formulated in work of the second author with Thomas Lam~\cite{LS14}.
It will feature prominently in their work, currently in preparation, on de Rham cohomology of cluster varieties.

\section{Postnikov diagrams and bounded affine permutations} \label{PositroidBackground}

Fix positive integers $k$ and $n$, with $0 < k < n$. We define a \newword{bounded affine permutation} of type $(k,n)$ to be a map $w: \ZZ \to \ZZ$ obeying the conditions
\begin{enumerate}
\item $i \leq w(i) \leq i+n$.
\item $w(i+n)=w(i)+n$, and so $w$ descends to a permutation of $\ZZ/n$.
\item $\frac{1}{n} \sum_{i=1}^{n} (w(i)-i) = k$.
\end{enumerate}
We will also write $w$ for the resulting permutation of $\ZZ/n$. 
The Grassmannian $G(k,n)$ is stratified into locally closed subvarieties called open positroid varieties, which are indexed by bounded affine permutations of type $(k,n)$.
See~\cite{Pos} and~\cite{KLS13}.

The \newword{length} of a bounded affine permutation $w$, denoted $\ell(w)$, is
\[ \# \{ 1 \leq i \leq n,\ i < j  \ :\  w(i) > w(j) \} \]
The unique bounded affine permutation of type $(k,n)$ of length $0$ is $x \mapsto x+k$; the maximum possible length of a  bounded affine permutation of type $(k,n)$ is $k(n-k)$, achieved by permutations which descend to the identity modulo $n$. 
Our proof of Theorem~\ref{MainTheorem} is by reverse induction on length, among other things.

Given a $k$-plane $L$ in $n$-space, we represent $L$ as the row span of a $k \times n$ matrix $M$. 
For integers $i \leq j \leq i+n$, we write $M_{ij}$ for the $k \times (j-i+1)$ matrix made up of the columns of $a$ whose positions lie in the interval $[i,j]$ modulo $n$.
We define $r_{ij}(L)$ to be the rank of $M_{ij}$; this quantity is independent of the choice of representative $L$ for $M$.
We define the \newword{positroid cell} $\Pio(w)$ to be the set of $L$ in $G(k,n)$ with $r_{ij}(L) = (j-i+1) - \#( [i,j] \cap w([i,j]) )$. 

Each $\Pio(w)$ is an irreducible, locally closed, algebraic variety. 
The dimension of $\Pio(w)$ is $k(n-k) - \ell(w)$.
See \cite{KLS13} for numerous alternate definitions of positroid cells.

It is anticipated that the homogeneous coordinate ring of an positroid cell is naturally a cluster algebra. We now describe how to construct seeds for this cluster algebra.
Let $w$ be a bounded affine permutation. A \newword{Postnikov diagram}\footnote{Postnikov uses the terms ``alternating strand diagram" for (1), (2) and (3), with the additional possibility that closed loops within the disc are allowed.  The absence of loops and condition (4) are equivalent to his ``leafless reduced'' condition on plabic graphs.  We follow Scott~\cite{Sco06} in terming these Postnikov diagrams.
} for $w$ consists of a disc with $n$ marked points $\ell_1$, $\ell_2$, \dots, $\ell_n$ around the boundary (in circular order) and $n$ directed paths called \newword{strands} within the disc, one from $i$ to $w(i)$ for each $i$ from $1$ to $n$. These strands obey the following conditions:
\begin{enumerate}
\item There are no triple crossings between strands. Any two strands cross transversely, and do so finitely many times.
\item If we follow any given strand, the other strands alternately cross it from the left and from the right.
\item No strand crosses itself, except that, if $\pi(i)=i$ or $i+n$, the strand leaving $\ell_i$ will return to $\ell_i$ again. If $\pi(i)=i$, we require that the strand from $\ell(i)$ to itself circle counter-clockwise, otherwise we require that it circle clockwise. 
\item If we consider any two strands $\gamma$ and $\delta$, with some finite list of intersection points, then they pass through their intersection points in opposite orders.
\end{enumerate}
We will refer to the strand from $w^{-1}(i)$ to $i$ as the $i$-th strand.  Some examples of Postnikov diagrams are given in Figure \ref{fig: PDexamples}.
 
\begin{figure}[h!t]
\begin{tikzpicture}
\begin{scope}[scale=.65]
		\draw (0,0) circle (4);
		\node[invisible] (1) at (180:4) {};
		\node[invisible] (2) at (120:4) {};
		\node[invisible] (3) at (60:4) {};
		\node[invisible] (4) at (0:4) {};
		\node[invisible] (5) at (-60:4) {};
		\node[invisible] (6) at (-120:4) {};
		
		\node[left] at (1) {$\ell_1$};
		\node[above left] at (2) {$\ell_2$};
		\node[above right] at (3) {$\ell_3$};
		\node[right] at (4) {$\ell_4$};
		\node[below right] at (5) {$\ell_5$};
		\node[below left] at (6) {$\ell_6$};

		\node[invisible] (a) at (165:2.5) {};
		\node[invisible] (b) at (135:2.5) {};
		\node[invisible] (c) at (45:2.5) {};
		\node[invisible] (d) at (15:2.5) {};
		\node[invisible] (e) at (180:1.25) {};
		\node[invisible] (f) at (120:1.25) {};
		\node[invisible] (g) at (60:1.25) {};
		\node[invisible] (h) at (0:1.25) {};
		\node[invisible] (i) at (300:1.25) {};
		\node[invisible] (j) at (240:1.25) {};
		\node[invisible] (k) at (255:2.5) {};
		\node[invisible] (l) at (285:2.5) {};

		\PostnikovWhiteBoundaryThree{a}{1}{b}{e};
		\PostnikovBlackBoundaryThree{b}{2}{f}{a};
		\PostnikovWhiteBoundaryThree{c}{3}{d}{g};
		\PostnikovBlackBoundaryThree{d}{4}{h}{c};
		\PostnikovBlack{e}{a}{f}{j};
		\PostnikovWhite{f}{b}{g}{e};
		\PostnikovBlack{g}{c}{h}{f};
		\PostnikovWhite{h}{g}{d}{i};
		\PostnikovBlack{i}{l}{j}{h};
		\PostnikovWhite{j}{e}{i}{k};
		\PostnikovBlackBoundaryThree{k}{6}{j}{l};
		\PostnikovWhiteBoundaryThree{l}{5}{k}{i};
		
		\node at (0,-6) {\shortstack{A Postnikov diagram for \\ $\left\{\begin{array}{cccccc}
		1 & 2 & 3 & 4 & 5 & 6 \\
		\downarrow & \downarrow & \downarrow & \downarrow & \downarrow & \downarrow \\
		4 & 5 & 6 & 7 & 8 & 9 \end{array}\right\}$}
		};
\end{scope}
\begin{scope}[xshift=2.5in,scale=.65]
		\draw (0,0) circle (4);
		\node[invisible] (1) at (180:4) {};
		\node[invisible] (2) at (120:4) {};
		\node[invisible] (3) at (60:4) {};
		\node[invisible] (4) at (0:4) {};
		\node[invisible] (5) at (-60:4) {};
		\node[invisible] (6) at (-120:4) {};
		
		\node[left] at (1) {$\ell_1$};
		\node[above left] at (2) {$\ell_2$};
		\node[above right] at (3) {$\ell_3$};
		\node[right] at (4) {$\ell_4$};
		\node[below right] at (5) {$\ell_5$};
		\node[below left] at (6) {$\ell_6$};
		
		\node[invisible] (a) at (-2.75,.25) {};
		\node[invisible] (b) at (-1.5,2) {};
		\node[invisible] (c) at (1.5,2) {};
		\node[invisible] (d) at (2.75,.25) {};
		\node[invisible] (e) at (-1.5,-1) {};
		\node[invisible] (f) at (-.35,.75) {};
		\node[invisible] (g) at (.35,.75) {};
		\node[invisible] (h) at (1.5,-1) {};
		\node[invisible] (i) at (.5,-2) {};
		\node[invisible] (j) at (-.5,-2) {};

		\PostnikovWhiteBoundaryThree{a}{1}{b}{e};
		\PostnikovBlackBoundaryThree{b}{2}{f}{a};
		\PostnikovWhiteBoundaryThree{c}{3}{d}{g};
		\PostnikovBlackBoundaryThree{d}{4}{h}{c};
		\PostnikovBlack{e}{a}{f}{j};
		\PostnikovWhite{f}{b}{g}{e};
		\PostnikovBlack{g}{c}{h}{f};
		\PostnikovWhite{h}{g}{d}{i};
		\PostnikovBlackBoundaryThree{i}{5}{j}{h};
		\PostnikovWhiteBoundaryThree{j}{6}{e}{i};
		
		\node at (0,-6) {\shortstack{A Postnikov diagram for \\ $\left\{\begin{array}{cccccc}
		1 & 2 & 3 & 4 & 5 & 6 \\
		\downarrow & \downarrow & \downarrow & \downarrow & \downarrow & \downarrow \\
		4 & 6 & 5 & 7 & 8 & 9 \end{array}\right\}$}
		};
\end{scope}
\end{tikzpicture}
\caption{Examples of Postnikov diagrams of type $(3,6)$}
\label{fig: PDexamples}
\end{figure}
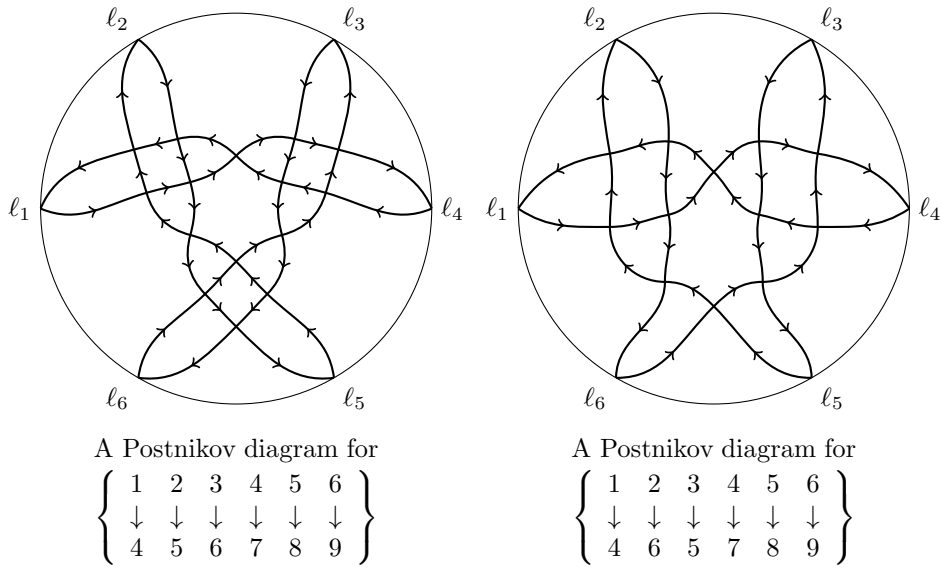

For future use, we note the following consequence of condition (4).

\begin{lemma} \label{nesters dont cross}
If $p < q < r < s$, and $D$ is a Postnikov diagram for a bounded affine permutation $w$, with $w(p) = s$ and $w(q) = r$, then the strands $p \to s$ and $q \to r$ in $w$ cannot cross.
\end{lemma}

\begin{proof}
We prove the contrapositive: Suppose that $D$ is a Postnikov diagram with strands $p \to s$ and $q \to r$ that do cross. If they cross an odd number of times then $p$ and $s$ are on opposite sides of the chord from $q$ to $r$, contradicting the proposed ordering.
If $p \to s$ and $q \to r$ cross an even number of times, then the boundary points must occur in circular order $(p,s,q,r)$ (see Figure \ref{fig: nocross}), again contradicting the proposed ordering.
\end{proof}

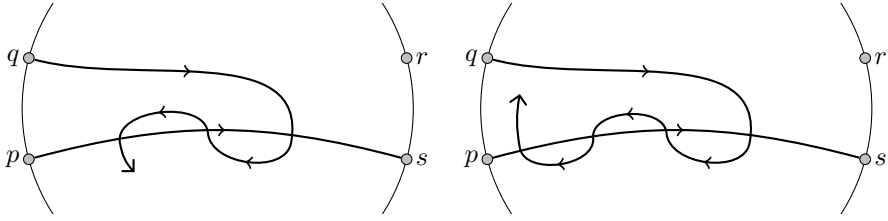
\begin{figure}[h!t]
\begin{tikzpicture}
\begin{scope}[scale=.65]
	\clip (-5,-2.15) rectangle (5,2.15);
	\draw (0,0) circle (4);
	\node[dot] (p) at (195:4) {};
	\node[dot] (q) at (165:4) {};
	\node[dot] (r) at (15:4) {};
	\node[dot] (s) at (-15:4) {};
	
	\node[left] at (p) {$p$};
	\node[left] at (q) {$q$};
	\node[right] at (r) {$r$};
	\node[right] at (s) {$s$};
		
	\draw[oriented,out=15,in=165] (p) to (s);
	\draw[oriented,out=-15,in=80] (q) to (1.5,-.7);
	\draw[oriented,out=-100,in=-85] (1.5,-.7) to (-.2,-.5);
	\draw[oriented,out=95,in=100] (-.2,-.5) to (-2,-.7);
	\draw[thick,-angle 90,out=-80] (-2,-.7) to (-1.7,-1.3);
\end{scope}
\begin{scope}[xshift=2.40in,scale=.65]
	\clip (-5,-2.15) rectangle (5,2.15);
	\draw (0,0) circle (4);
	\node[dot] (p) at (195:4) {};
	\node[dot] (q) at (165:4) {};
	\node[dot] (r) at (15:4) {};
	\node[dot] (s) at (-15:4) {};
	
	\node[left] at (p) {$p$};
	\node[left] at (q) {$q$};
	\node[right] at (r) {$r$};
	\node[right] at (s) {$s$};
		
	\draw[oriented,out=15,in=165] (p) to (s);
	\draw[oriented,out=-15,in=80] (q) to (1.5,-.7);
	\draw[oriented,out=-100,in=-85] (1.5,-.7) to (-.2,-.5);
	\draw[oriented,out=95,in=95] (-.2,-.5) to (-1.7,-.6);
	\draw[oriented,out=-85,in=-80] (-1.7,-.6) to (-3.2,-.8);
	\draw[thick,-angle 90,out=100,in=-100] (-3.2,-.8) to (-3.2,.3);
\end{scope}
\end{tikzpicture}
\caption{Parallel strands cannot cross (Lemma \ref{nesters dont cross})}
\label{fig: nocross}
\end{figure}

The strands of a Postnikov diagram divide the disc into three sorts of regions: \newword{clockwise regions}, where all the strands on their boundary circle clockwise, \newword{counterclockwise regions}, defined analogously, and \newword{alternating regions}, where the adjacent strands alternate directions (see Figure \ref{fig: threeregions}).  All boundary regions are defined to be alternating.  Condition $(2)$ above is equivalent to requiring that every region is either clockwise, counterclockwise or alternating.

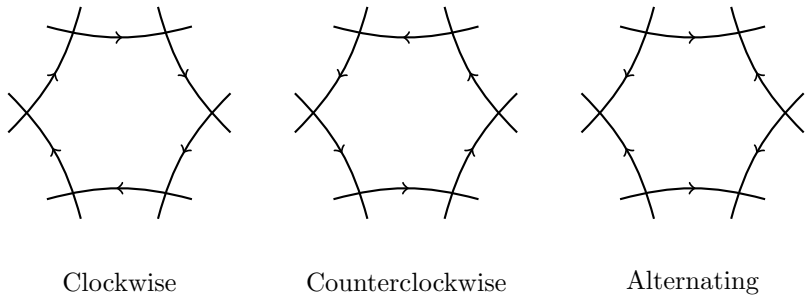
\begin{figure}[h!t]
\begin{tikzpicture}
\begin{scope}[xshift=-1.5in,scale=.75]
	\draw[oriented] (70:2) to [relative, out=-15,in=-165] (-10:2);
	\draw[oriented] (130:2) to [relative, out=-15,in=-165] (50:2);
	\draw[oriented] (190:2) to [relative, out=-15,in=-165] (110:2);
	\draw[oriented] (250:2) to [relative, out=-15,in=-165] (170:2);
	\draw[oriented] (310:2) to [relative, out=-15,in=-165] (230:2);
	\draw[oriented] (370:2) to [relative, out=-15,in=-165] (290:2);
	
	\node at (0,-3) {Clockwise};
\end{scope}
\begin{scope}[scale=.75]
	\draw[oriented] (-10:2) to [relative, out=15,in=165] (70:2);
	\draw[oriented] (50:2) to [relative, out=15,in=165] (130:2);
	\draw[oriented] (110:2) to [relative, out=15,in=165] (190:2);
	\draw[oriented] (170:2) to [relative, out=15,in=165] (250:2);
	\draw[oriented] (230:2) to [relative, out=15,in=165] (310:2);
	\draw[oriented] (290:2) to [relative, out=15,in=165] (370:2);
	
	\node at (0,-3) {Counterclockwise};
\end{scope}
\begin{scope}[xshift=1.5in,scale=.75]
	\draw[oriented] (-10:2) to [relative, out=15,in=165] (70:2);
	\draw[oriented] (130:2) to [relative, out=-15,in=-165] (50:2);
	\draw[oriented] (110:2) to [relative, out=15,in=165] (190:2);
	\draw[oriented] (250:2) to [relative, out=-15,in=-165] (170:2);
	\draw[oriented] (230:2) to [relative, out=15,in=165] (310:2);
	\draw[oriented] (370:2) to [relative, out=-15,in=-165] (290:2);
	
	\node at (0,-3) {Alternating};
\end{scope}
\end{tikzpicture}
\caption{Three types of region in a Postnikov diagram}
\label{fig: threeregions}
\end{figure}

A Postnikov diagram determines an ice quiver as follows.  Form a quiver $\tQ$ whose vertices are the alternating regions of the Postnikov diagram, and where there is an edge between two vertices if the corresponding regions $R_1$ and $R_2$ are separated by a pair of strands which cross on the borders of $R_1$ and $R_2$.
Note that the faces of $\tQ$ are in bijection with the clockwise and counterclockwise regions of the Postnikov diagram; orient $\tQ$ by directing edges clockwise around the clockwise regions of the diagram and counterclockwise around the counterclockwise regions.  Finally, make $\tQ$ an ice quiver by freezing the vertices corresponding to boundary regions.  An example is given in Figure \ref{fig: PDquivers}.

\def\faint{draw=black!50}
\def\salient{draw=black}

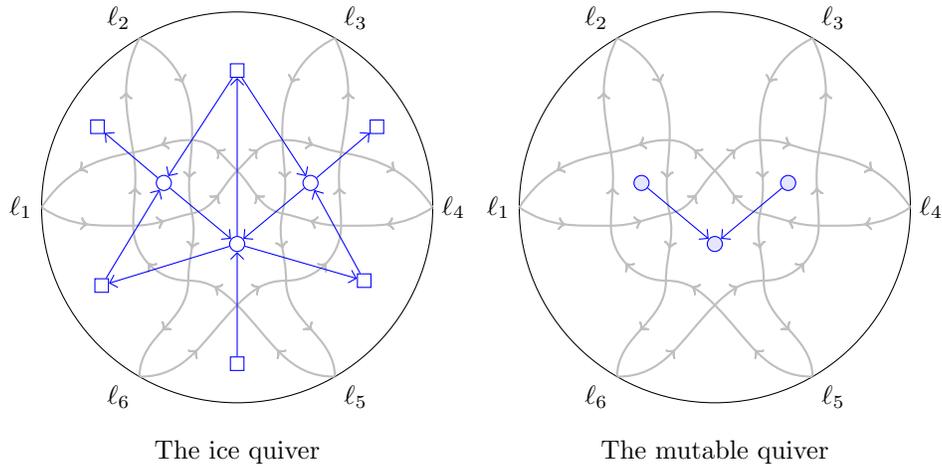
\begin{figure}[h!t]
\begin{tikzpicture}
\begin{scope}[scale=.65]
		\draw (0,0) circle (4);
		\node[invisible] (1) at (180:4) {};
		\node[invisible] (2) at (120:4) {};
		\node[invisible] (3) at (60:4) {};
		\node[invisible] (4) at (0:4) {};
		\node[invisible] (5) at (-60:4) {};
		\node[invisible] (6) at (-120:4) {};
		
		\node[left] at (1) {$\ell_1$};
		\node[above left] at (2) {$\ell_2$};
		\node[above right] at (3) {$\ell_3$};
		\node[right] at (4) {$\ell_4$};
		\node[below right] at (5) {$\ell_5$};
		\node[below left] at (6) {$\ell_6$};
		
		\node[invisible] (a) at (-2.75,.25) {};
		\node[invisible] (b) at (-1.5,2) {};
		\node[invisible] (c) at (1.5,2) {};
		\node[invisible] (d) at (2.75,.25) {};
		\node[invisible] (e) at (-1.5,-1) {};
		\node[invisible] (f) at (-.35,.75) {};
		\node[invisible] (g) at (.35,.75) {};
		\node[invisible] (h) at (1.5,-1) {};
		\node[invisible] (i) at (.5,-2) {};
		\node[invisible] (j) at (-.5,-2) {};

		\PostnikovWhiteBoundaryThree[faded oriented]{a}{1}{b}{e};
		\PostnikovBlackBoundaryThree[faded oriented]{b}{2}{f}{a};
		\PostnikovWhiteBoundaryThree[faded oriented]{c}{3}{d}{g};
		\PostnikovBlackBoundaryThree[faded oriented]{d}{4}{h}{c};
		\PostnikovBlack[faded oriented]{e}{a}{f}{j};
		\PostnikovWhite[faded oriented]{f}{b}{g}{e};
		\PostnikovBlack[faded oriented]{g}{c}{h}{f};
		\PostnikovWhite[faded oriented]{h}{g}{d}{i};
		\PostnikovBlackBoundaryThree[faded oriented]{i}{5}{j}{h};
		\PostnikovWhiteBoundaryThree[faded oriented]{j}{6}{e}{i};
				
		\node[frozen,blue] (F1) at (210:3.2) {};
		\node[frozen,blue] (F2) at (150:3.3) {};
		\node[frozen,blue] (F3) at (90:2.8) {};
		\node[frozen,blue] (F4) at (30:3.3) {};
		\node[frozen,blue] (F5) at (-30:3) {};
		\node[frozen,blue] (F6) at (-90:3.2) {};
		
		\node[mutable,blue] (M1) at (-1.5,.5) {};
		\node[mutable,blue] (M2) at (0,-.75) {};
		\node[mutable,blue] (M3) at (1.5,.5) {};
		
		\draw[blue,-angle 90] (F1) to (M1);
		\draw[blue,-angle 90] (F3) to (M1);
		\draw[blue,-angle 90] (F3) to (M3);
		\draw[blue,-angle 90] (F5) to (M3);
		\draw[blue,-angle 90] (F6) to (M2);
		\draw[blue,-angle 90] (M1) to (F2);
		\draw[blue,-angle 90] (M1) to (M2);
		\draw[blue,-angle 90] (M2) to (F1);
		\draw[blue,-angle 90] (M2) to (F3);
		\draw[blue,-angle 90] (M2) to (F5);
		\draw[blue,-angle 90] (M3) to (M2);
		\draw[blue,-angle 90] (M3) to (F4);
		
		\node at (0,-5) {The ice quiver
		};
\end{scope}
\begin{scope}[xshift=2.5in,scale=.65]
		\draw (0,0) circle (4);
		\node[invisible] (1) at (180:4) {};
		\node[invisible] (2) at (120:4) {};
		\node[invisible] (3) at (60:4) {};
		\node[invisible] (4) at (0:4) {};
		\node[invisible] (5) at (-60:4) {};
		\node[invisible] (6) at (-120:4) {};
		
		\node[left] at (1) {$\ell_1$};
		\node[above left] at (2) {$\ell_2$};
		\node[above right] at (3) {$\ell_3$};
		\node[right] at (4) {$\ell_4$};
		\node[below right] at (5) {$\ell_5$};
		\node[below left] at (6) {$\ell_6$};
		
		\node[invisible] (a) at (-2.75,.25) {};
		\node[invisible] (b) at (-1.5,2) {};
		\node[invisible] (c) at (1.5,2) {};
		\node[invisible] (d) at (2.75,.25) {};
		\node[invisible] (e) at (-1.5,-1) {};
		\node[invisible] (f) at (-.35,.75) {};
		\node[invisible] (g) at (.35,.75) {};
		\node[invisible] (h) at (1.5,-1) {};
		\node[invisible] (i) at (.5,-2) {};
		\node[invisible] (j) at (-.5,-2) {};

		\PostnikovWhiteBoundaryThree[faded oriented]{a}{1}{b}{e};
		\PostnikovBlackBoundaryThree[faded oriented]{b}{2}{f}{a};
		\PostnikovWhiteBoundaryThree[faded oriented]{c}{3}{d}{g};
		\PostnikovBlackBoundaryThree[faded oriented]{d}{4}{h}{c};
		\PostnikovBlack[faded oriented]{e}{a}{f}{j};
		\PostnikovWhite[faded oriented]{f}{b}{g}{e};
		\PostnikovBlack[faded oriented]{g}{c}{h}{f};
		\PostnikovWhite[faded oriented]{h}{g}{d}{i};
		\PostnikovBlackBoundaryThree[faded oriented]{i}{5}{j}{h};
		\PostnikovWhiteBoundaryThree[faded oriented]{j}{6}{e}{i};
		
		\node[mutable,blue, fill=blue!10] (M1) at (-1.5,.5) {};
		\node[mutable,blue, fill=blue!10] (M2) at (0,-.75) {};
		\node[mutable,blue, fill=blue!10] (M3) at (1.5,.5) {};	

		\draw[blue,-angle 90] (M1) to (M2);
		\draw[blue,-angle 90] (M3) to (M2);
		
		\node at (0,-5) {The mutable quiver
		};
\end{scope}
\end{tikzpicture}
\caption{The ice quiver and mutable quiver of a Postnikov diagram}
\label{fig: PDquivers}
\end{figure}


We have the following theorem of Postnikov:

\begin{theorem}\cite[Corollary 14.2]{Pos}
For any bounded affine permutation $w$, there exists a Postnikov diagram $D$. If $D$ and $D'$ are two Postnikov diagrams for the same $w$, then it is possible to change $D$ to $D'$ by a sequence of square moves (defined in~\cite{Pos}), which have the effect of mutating the associated ice quiver.
\end{theorem}

It therefore makes sense, for any bounded affine permutation $w$, to define $\cA(w)$ to be the cluster algebra defined by the ice quiver of any Postnikov diagram for $w$.
 
 Our main result is
 \begin{theorem} \label{MainTheorem}
 For any bounded affine permutation $w$, the cluster algebra $\cA(w)$ is Louise.
 \end{theorem}
 
We now describe the expected connection between $\cA(w)$ and $\Pio(w)$.  A strand in a Postnikov diagram will divide the disc into two pieces, called its `left' and `right'.  
We will label each alternating region with the indices of the strands to whose right it lies.
This has the effect of labeling each alternating region with a $k$-element subset of $\{ 1,2, \ldots, n \}$, and each alternating region receives a distinct label (see Figure \ref{fig: PDPluckers}).

The corresponding Pl\"ucker coordinates are expected to form a cluster on $\Pio(w)$.

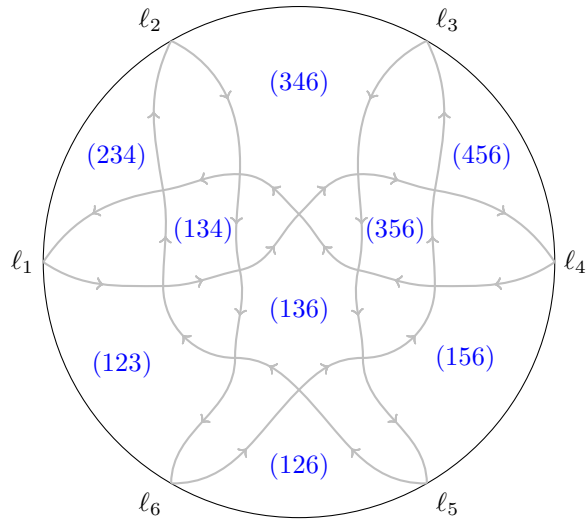
\begin{figure}[h!t]
\begin{tikzpicture}
\begin{scope}[scale=.85]
		\draw (0,0) circle (4);
		\node[invisible] (1) at (180:4) {};
		\node[invisible] (2) at (120:4) {};
		\node[invisible] (3) at (60:4) {};
		\node[invisible] (4) at (0:4) {};
		\node[invisible] (5) at (-60:4) {};
		\node[invisible] (6) at (-120:4) {};
		
		\node[left] at (1) {$\ell_1$};
		\node[above left] at (2) {$\ell_2$};
		\node[above right] at (3) {$\ell_3$};
		\node[right] at (4) {$\ell_4$};
		\node[below right] at (5) {$\ell_5$};
		\node[below left] at (6) {$\ell_6$};
		
		\node[invisible] (a) at (-2.75,.25) {};
		\node[invisible] (b) at (-1.5,2) {};
		\node[invisible] (c) at (1.5,2) {};
		\node[invisible] (d) at (2.75,.25) {};
		\node[invisible] (e) at (-1.5,-1) {};
		\node[invisible] (f) at (-.35,.75) {};
		\node[invisible] (g) at (.35,.75) {};
		\node[invisible] (h) at (1.5,-1) {};
		\node[invisible] (i) at (.5,-2) {};
		\node[invisible] (j) at (-.5,-2) {};

		\PostnikovWhiteBoundaryThree[faded oriented]{a}{1}{b}{e};
		\PostnikovBlackBoundaryThree[faded oriented]{b}{2}{f}{a};
		\PostnikovWhiteBoundaryThree[faded oriented]{c}{3}{d}{g};
		\PostnikovBlackBoundaryThree[faded oriented]{d}{4}{h}{c};
		\PostnikovBlack[faded oriented]{e}{a}{f}{j};
		\PostnikovWhite[faded oriented]{f}{b}{g}{e};
		\PostnikovBlack[faded oriented]{g}{c}{h}{f};
		\PostnikovWhite[faded oriented]{h}{g}{d}{i};
		\PostnikovBlackBoundaryThree[faded oriented]{i}{5}{j}{h};
		\PostnikovWhiteBoundaryThree[faded oriented]{j}{6}{e}{i};
		
		\node[blue] at (210:3.2) {$(123$)};
		\node[blue] at (150:3.3) {$(234$)};
		\node[blue] at (90:2.8) {$(346$)};
		\node[blue] at (30:3.3) {$(456$)};
		\node[blue] at (-30:3) {$(156$)};
		\node[blue] at (-90:3.2) {$(126$)};
		
		\node[blue] at (-1.5,.5) {$(134)$};
		\node[blue] at (0,-.75) {$(136)$};
		\node[blue] at (1.5,.5) {$(356)$};
		
\end{scope}
\end{tikzpicture}
\caption{The Pl\"ucker coordinates determined by a Postnikov diagram}
\label{fig: PDPluckers}
\end{figure}

Let $\widetilde{\mathcal{O}}(\Pio(w))$ be the homogeneous coordinate ring of $\Pio(w)$ for the Pl\"ucker embeding of the Grassmannian.  The following conjecture has circulated since Postnikov introduced his diagrams.

\begin{conjecture}\label{conj}
There is an isomorphism of rings
\[ \cA(w) \stackrel{\sim}{\longrightarrow} \widetilde{\mathcal{O}}(\Pio(w))\]
such that for any Postnikov diagram $D$ for $w$, the cluster variable $x_F$ in $\cA(w)$ corresponding to a face $F$ is sent to the Pl\"ucker coordinate $\Delta_I$, where $I$ is the label of the face $F$.
\end{conjecture}

\noindent It would follow that the face labels of a Postnikov diagram give a cluster of Pl\"ucker coordinates.

This conjecture generalizes the known cluster structure on the Grassmannian, as follows.  Let $\pi(k,n)$ be the map $x \mapsto x+k$, considered as a bounded affine permutation of type $(k,n)$. Scott~\cite{Sco06} showed that $\cA(\pi(k,n))$, without the frozen variables inverted, is the homogenous coordinate ring of the Grassmannian. Using the definition of cluster algebra in this paper, where frozen variables are inverted, this implies a special case of the conjecture, that $\cA(\pi(k,n))$ is $\widetilde{\mathcal{O}}(\Pio(\pi(k,n)))$. Thus, Theorem \ref{MainTheorem} implies that $\Pio(\pi(k,n))$ is locally acyclic.

\begin{remark}
Instead of labeling a strand with the index of its target, we could have labeled it with the index of its source.  If we still label each alternating region with the set of strands of which it is on the left side, we obtain a different collection of face labels and corresponding Pl\"ucker coordinates.  There is then a \emph{source-labeled} version of Conjecture \ref{conj}, which says there is a different isomorphism 
\[ \cA(w) \stackrel{\sim}{\longrightarrow} \widetilde{\mathcal{O}}(\Pio(w))\]
which sends the cluster variable of a face $F$ to the Pl\"ucker coordinate corresponding to this alternate labeling.  

The conjectures can be regarded as giving $\widetilde{\mathcal{O}}(\Pio(w))$ two potentially distinct cluster structures.  When $w$ is $\pi(k,n)$, these cluster structures coincide; however, there are other bounded affine permutation for which these two cluster structures have mutation-inequivalent clusters.  This includes some double Bruhat cells; in the case of $GL_4^{e,(14)}$, the target-labeling convention gives the cluster structure from \cite{BFZ05}, while the source-labeling convention gives the image of this cluster structure under the adjugate map.  A cluster for either cluster structure defines an open algebraic torus inside $\Pio(w)$.   Computations suggest that sets of algebraic tori in $\Pio(w)$ coming from either cluster structure coincide, so perhaps they are the `same' cluster structure in an appropriately generalized sense.

This source-labeled cluster algebra appears to be closer to the one recently constructed by Leclerc~\cite{Lec14}, but the details are not yet clear.
\end{remark}

\section{Proof of the main result}\label{Main}

We will now prove the main theorem, that $\cA(w)$ is Louise.  We will use the following lemmas to reduce from one permutation to another. We write $s_i$ for the affine permutation 
\[ s_i(j) = \begin{cases} j+1 & j \equiv i \bmod n \\ j-1 & j \equiv i+1 \bmod n \\ j & \mbox{otherwise} \end{cases}. \]

\begin{lemma} \label{fixed point}
Let $w$ be a bounded affine permutation of type $(k,n)$ with $w(i) = i$ (respectively $w(i) = i+n$). Define a permutation $w'$ of type $(k,n-1)$ (respectively $(k-1,n-1)$) as follows:  Define $Y$ to be the set of integers which are \emph{not} $\equiv i \bmod n$. Then $w$ maps $Y$ bijectively to itself; choose an order preserving bijection $\alpha : \ZZ \to Y$, and define $w' = \alpha^{-1} \circ w \circ \alpha$.

Then $w$ and $w'$ have the same exchange type.
\end{lemma}

\begin{proof}
An alternating strand diagram for $w$ is obtained from one for $w'$ by adding a self loop at $i$, not crossing any of the other strands. This clearly does not alter the quiver.
\end{proof}

\begin{lemma} \label{short arc}
Let $w$ be a bounded affine permutation with $w(i)=i+1$ or $w(i+1)=i+n$. If $n\geq2$, then $s_i w$ and $w s_i$ are also bounded affine permutations, and $\cA(s_i w)$, $\cA(w s_i)$ and $\cA(w)$ all have the same exchange type.
\end{lemma}

\begin{proof}
We give the proof in the case that $w(i)=i+1$; the case of $w(i+1)=i+n$ is essentially identical.
We consider the action of $s_iw$.
\[ s_iw(j) = \left\{ \begin{array}{ll}
w(j)+1 & \text{if }w(j)\equiv i\bmod n \\
w(j)-1=j & \text{if }w(j)\equiv i+1\bmod n,\ \text{equivalently},\ \text{if}\ j \equiv i \bmod n \\
w(j) & \text{otherwise}
\end{array}\right\} \]
We observe that $j\leq s_iw(j)\leq j+n$, unless that $w(j)\equiv i\bmod n$ and $w(j)=j+n$.  In that case, $j\equiv i\bmod n$, and so $w(j)=j+1$. So $j+1=j+n$ and $n=1$, contrary to our assumption $n \geq 2$.


We next consider the action of $ws_i$.
\[ ws_i(j) = \left\{ \begin{array}{ll}
w(j+1) & \text{if }j\equiv i\bmod n \\
w(j-1)=j & \text{if }j\equiv i+1\bmod n \\
w(j) & \text{otherwise}
\end{array}\right\} \]
We observe that $j\leq ws_i(j)\leq j+n$, unless $j\equiv i\bmod n$ and $w(j+1)=j+n+1$.   In that case, $w(j+1)\equiv i+1 \bmod n$.  
Since $w$ is a bijection and $n$-periodic, this implies that $j+1\equiv i\mod n$. So $j \equiv j+1 \bmod n$ and, once again, we contradict the assumption that $n \geq 2$.
This completes the verification that $s_iw$ and $ws_i$ are bounded affine permutations when $n\geq2$.


Let $D$ be a Postnikov diagram for $w$ which has strand from $\ell_i$ to $\ell_{i+1}$ which does not cross any other strands.  The strands in a small neighborhood of the points $\ell_{i-1},\ell_i$, and $\ell_{i+1}$ may be reconnected as in Figure \ref{fig: shortstrand} to produce Postnikov diagrams for $s_iw$ and $ws_i$ with the same mutable quiver as $D$.  Note that the ice quivers for $s_iw$ and $ws_i$ have one fewer frozen vertex than $w$.
%
\end{proof}

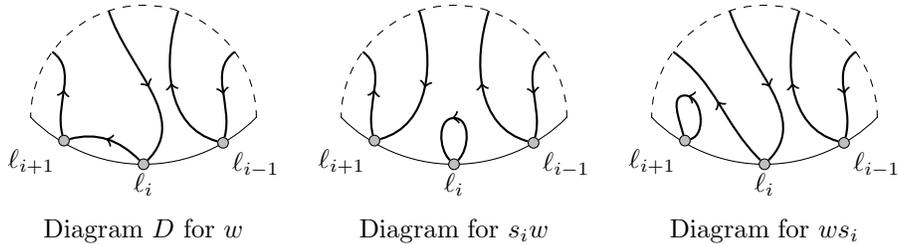
\begin{figure}[h!t]

\begin{tikzpicture}
\begin{scope}[scale=.75,baseline=(current bounding box.center)]
	\draw[dashed] (2,0) arc (0:180:2);
	\begin{scope}
		\clip (0,0) circle (2);
		\draw (0,2) circle (2.82);
	\end{scope}
	\node[dot] (a) at (1.41, -.44) {};
	\node[dot] (b) at (0,-.82) {};
	\node[dot] (c) at (-1.41, -.4) {};
	\node[below right] at (a) {$\ell_{i-1}$};
	\node[below] at (b) {$\ell_i$};
	\node[below left] at (c) {$\ell_{i+1}$};
	\draw[oriented, draw=black] (c) to [out=105,in=-36] (144:2);
	\draw[oriented, draw=black] (b) to [out=135,in= 15] (c);
	\draw[oriented, draw=black] (108:2) to [out=-72,in=45] (b);
	\draw[oriented, draw=black] (a) to [out=165,in=-108] (72:2);
	\draw[oriented, draw=black] (36:2) to [out=-144,in=75] (a);
	\node at (0,-2) {Diagram $D$ for $w$};
\end{scope}
\begin{scope}[xshift=1.625in,scale=.75,baseline=(current bounding box.center)]
	\draw[dashed] (2,0) arc (0:180:2);
	\begin{scope}
		\clip (0,0) circle (2);
		\draw (0,2) circle (2.82);
	\end{scope}
	\node[dot] (a) at (1.41, -.44) {};
	\node[dot] (b) at (0,-.82) {};
	\node[dot] (c) at (-1.41, -.4) {};
	\node[below right] at (a) {$\ell_{i-1}$};
	\node[below] at (b) {$\ell_i$};
	\node[below left] at (c) {$\ell_{i+1}$};
	\draw[oriented, draw=black] (c) to [out=105,in=-36] (144:2);
	\draw[oriented, draw=black] (108:2) to [out=-72,in=15] (c);
	\draw[oriented, draw=black] (b) to [out=45,in=0] (0,0) to [out=180,in=135] (b);
	\draw[oriented, draw=black] (a) to [out=165,in=-108] (72:2);
	\draw[oriented, draw=black] (36:2) to [out=-144,in=75] (a);
	
	\node at (0,-2) {Diagram for $s_iw$};
\end{scope}
\begin{scope}[xshift=3.25in,scale=.75,baseline=(current bounding box.center)]
	\draw[dashed] (2,0) arc (0:180:2);
	\begin{scope}
		\clip (0,0) circle (2);
		\draw (0,2) circle (2.82);
	\end{scope}
	\node[dot] (a) at (1.41, -.44) {};
	\node[dot] (b) at (0,-.82) {};
	\node[dot] (c) at (-1.41, -.4) {};
	\node[below right] at (a) {$\ell_{i-1}$};
	\node[below] at (b) {$\ell_i$};
	\node[below left] at (c) {$\ell_{i+1}$};
	\draw[oriented, draw=black] (b) to [out=135,in=-36] (144:2);
	\draw[oriented, draw=black] (108:2) to [out=-72,in=45] (b);
	\draw[oriented, draw=black] (c) to [out=15,in=0] (-1.41,.4) to [out=180,in=105] (c);
	\draw[oriented, draw=black] (a) to [out=165,in=-108] (72:2);
	\draw[oriented, draw=black] (36:2) to [out=-144,in=75] (a);
	
	\node at (0,-2) {Diagram for $ws_i$};
\end{scope}
\end{tikzpicture} 
\caption{Diagrams with the same mutable quiver, for $w(i)=i+1$}
\label{fig: shortstrand}
\end{figure}



We will apply the next lemma in two different settings, which are described by the corollaries that follow it.

\begin{lemma} \label{covers}
Let $v$ be a bounded affine permutation with $v^{-1}(i) = a$, $v^{-1}(i+1) = b$, $v(i) = c$ and $v(i+1)=d$. Assume that none of $\{ a,b,c,d \}$ are congruent to any of $i$, $i+1$ modulo $n$. Assume further that $a<b$ and $c<d$. 

With this notation, $s_i v$, $v s_i$ and $s_i v s_i$ are all bounded affine permutations.
The cluster algebras $\cA(s_i v)$ and $\cA(v s_i)$ have the same exchange type. 
Furthermore, one of the following two scenarios occurs (as illustrated in Figure \ref{fig: twocases}):
\begin{enumerate}
\item There is a mutable quiver $\Q$ for $\cA(v)$ which has a vertex $x$ with an arrow to its unique neighbor $y$, so that $\Q[x^{-1}]$ is a mutable quiver for $\cA(s_i v)$; $\Q[y^{-1}]$ is the disjoint union of a mutable quiver for $\cA(s_i v s_i)$ and an isolated point; and $\Q[x^{-1}, y^{-1}]$ is a mutable quiver for $\cA(s_i v s_i)$.
\item There is a mutable quiver $\Q$ for $\cA(v)$ which is the disjoint union of a point and a mutable quiver for $\cA(s_i v)$. 
\end{enumerate}
\end{lemma}

\begin{proof}
We first check that $s_i v$, $v s_i$ and $s_i v s_i$ are bounded affine permutations. We need to check the inequalities $a \leq i+1 \leq a+n$, $b \leq i \leq b+n$, $i+1 \leq c \leq i+1+n$ and  $i \leq d \leq i+n$. Since we know $a \leq i \leq a+n$, $b \leq i+1 \leq b+n$,  $i \leq c \leq i+n$ and $i+1 \leq d \leq i+1+n$, we just need to check that $i \neq a+n$, $b \neq i+1$, $i \neq c$ and $d \neq i+1+n$, respectively. The failure of any of these inequalities would produce a collision between $\{ a, b, c, d \}$ and $\{ i, i+1 \}$ modulo $n$.

Let $D$ be a Postnikov diagram for $s_i v s_i$. Note that the strands $a \to i+1$ and $b \to i$ of $D$ cannot cross, by Lemma~\ref{nesters dont cross}, and the same holds for the strands $i+1 \to c$ and $i \to d$. 

Extend diagram $D$ to three larger diagrams as shown in Figure \ref{fig: fourdiagrams}. It is obvious that these larger diagrams obey conditions $(1)$, $(2)$ and $(3)$ in the definition of a Postnikov diagram.  By the observation of the above paragraph, the new crossings added by these diagrams involve strands which don't cross in $D$, so condition $(4)$ holds as well, and these larger diagrams are Postnikov diagrams. We can easily check that they have connectivity $s_i v$, $v s_i$ and $v$ respectively. 

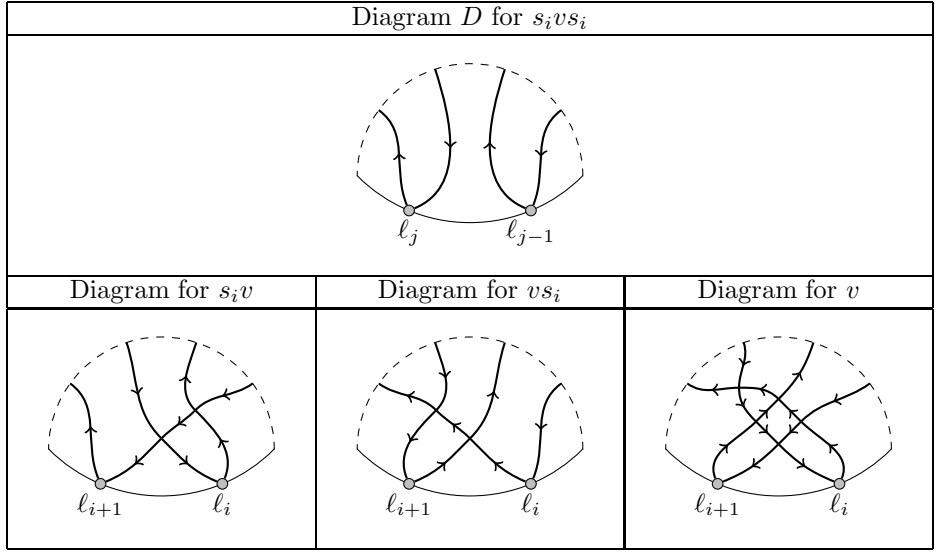
\begin{figure}[h!t]
\begin{tabular}{|c|c|c|}
\hline
\multicolumn{3}{|c|}{\text{Diagram $D$ for $s_ivs_i$} } \\
\hline
\multicolumn{3}{|c|}{
\begin{tikzpicture}[scale=.75,baseline=(current bounding box.center)]
	\path[use as bounding box] (-2.5,-1.75) rectangle (2.5,2.5);
	\draw[dashed] (2,0) arc (0:180:2);
	\begin{scope}
		\clip (0,0) circle (2);
		\draw (0,2) circle (2.82);
	\end{scope}
	\node[dot] (a) at (1.08, -.605) {};
	\node[dot] (b) at (-1.08, -.605) {};
	\node[below] at (a) {$\ell_{j-1}$};
	\node[below] at (b) {$\ell_j$};
	\draw[oriented, draw=black] (b) to [out=112.5,in=-36] (144:2);
	\draw[oriented, draw=black] (108:2) to [out=-72,in=22.5] (b);
	\draw[oriented, draw=black] (a) to [out=157.5,in=-108] (72:2);
	\draw[oriented, draw=black] (36:2) to [out=-144,in=67.5] (a);
\end{tikzpicture}
} \\
\hline
Diagram for $s_iv$ & Diagram for $vs_i$ & Diagram for $v$ \\
\hline
\begin{tikzpicture}[scale=.75,baseline=(current bounding box.center)]
	\path[use as bounding box] (-2.5,-1.75) rectangle (2.5,2.5);
	\draw[dashed] (2,0) arc (0:180:2);
	\begin{scope}
		\clip (0,0) circle (2);
		\draw (0,2) circle (2.82);
	\end{scope}
	\node[dot] (a) at (1.08, -.605) {};
	\node[dot] (b) at (-1.08, -.605) {};
	\node[below] at (a) {$\ell_{i}$};
	\node[below] at (b) {$\ell_{i+1}$};
	\node[invisible] (z) at (.6,.7) {};
	\node[invisible] (t) at (0,.2) {};

	\draw[oriented] (b) to [out=112.5,in=-36] (144:2);

	\draw[oriented] (108:2) to [out=-72,in=135] (t);
	\draw[oriented] (t) to [out=-45,in=157.5] (a);
	
	\draw[oriented] (a) to [out=67.5,in=-45] (z);
	\draw[oriented] (z) to [out=135,in=-108] (72:2);

	\draw[oriented] (36:2) to [out=-144,in=45] (z);
	\draw[oriented] (z) to [out=225,in=45] (t);
	\draw[oriented] (t) to [out=225,in=22.5] (b);
\end{tikzpicture}
&
\begin{tikzpicture}[scale=.75,baseline=(current bounding box.center)]
	\path[use as bounding box] (-2.5,-1.75) rectangle (2.5,2.5);
	\draw[dashed] (2,0) arc (0:180:2);
	\begin{scope}
		\clip (0,0) circle (2);
		\draw (0,2) circle (2.82);
	\end{scope}
	\node[dot] (a) at (1.08, -.605) {};
	\node[dot] (b) at (-1.08, -.605) {};
	\node[below] at (a) {$\ell_{i}$};
	\node[below] at (b) {$\ell_{i+1}$};
	\node[invisible] (x) at (0,.2) {};
	\node[invisible] (u) at (-.6,.7) {};

	\draw[oriented] (b) to [out=22.5,in=225] (x);
	\draw[oriented] (x) to [out=45,in=-108] (72:2);
	
	\draw[oriented] (108:2) to [out=-72,in=45] (u);
	\draw[oriented] (u) to [out=225,in=112.5] (b);
	
	\draw[oriented] (a) to [out=157.5,in=-45] (x);
	\draw[oriented] (x) to [out=135,in=-45] (u);
	\draw[oriented] (u) to [out=135,in=-36] (144:2);
	
	\draw[oriented] (36:2) to [out=-144,in=67.5] (a);
\end{tikzpicture}
&
\begin{tikzpicture}[scale=.75,baseline=(current bounding box.center)]
	\path[use as bounding box] (-2.5,-1.75) rectangle (2.5,2.5);
	\draw[dashed] (2,0) arc (0:180:2);
	\begin{scope}
		\clip (0,0) circle (2);
		\draw (0,2) circle (2.82);
	\end{scope}
	\node[dot] (a) at (1.08, -.605) {};
	\node[dot] (b) at (-1.08, -.605) {};
	\node[below] at (a) {$\ell_{i}$};
	\node[below] at (b) {$\ell_{i+1}$};
	\node[invisible] (x) at (0,.9) {};
	\node[invisible] (y) at (-.4,.5) {};
	\node[invisible] (z) at (.4,.5) {};
	\node[invisible] (t) at (0,.1) {};
	\node[invisible] (u) at (-.7,1.1) {};

	\draw[oriented] (b) to [out=112.5,in=225] (y);
	\draw[oriented] (y) to [out=45,in=225] (x);
	\draw[oriented] (x) to [out=45,in=-108] (72:2);
	
	\draw[oriented] (108:2) to [out=-72,in=90] (u);
	\draw[oriented] (u) to [out=270,in=135] (y);
	\draw[oriented] (y) to [out=-45,in=135] (t);
	\draw[oriented] (t) to [out=-45,in=157.5] (a);
	
	\draw[oriented] (a) to [out=67.5,in=-45] (z);
	\draw[oriented] (z) to [out=135,in=-45] (x);
	\draw[oriented] (x) to [out=135,in=0] (u);
	\draw[oriented] (u) to [out=180,in=-36] (144:2);
	
	\draw[oriented] (36:2) to [out=-144,in=45] (z);
	\draw[oriented] (z) to [out=225,in=45] (t);
	\draw[oriented] (t) to [out=225,in=22.5] (b);
\end{tikzpicture}
\\
\hline
\end{tabular}
\caption{Building Postnikov diagrams for $s_iv$, $vs_i$ and $s_ivs_i$}
\label{fig: fourdiagrams}
\end{figure}

Let $\Q$ be the quiver corresponding to this Postnikov diagram for $v$, and let $x$ and $y$ be the vertices marked in Figure \ref{fig: fourdiagrams}.We see that $x$ is a mutable vertex of $\Q$, but $y$ may either be a mutable vertex or a frozen vertex, and that all the other neighbors of $x$ are frozen. (Figure \ref{fig: twocases} shows how either case may occur.)
Moreover, the mutable part of $\Q[x^{-1}]$ is isomorphic to both the mutable part of the quiver we constructed for $s_i v$ and for $v s_i$, so we have verified that $\cA(s_i v)$ and $\cA(v s_i)$ have the same exchange type.

\begin{figure}[h!b]
\begin{tikzpicture}
\begin{scope}[xshift=2.3in,yshift=.7cm,scale=.75]
	\draw (-45:2) arc (-45:-135:2);
	\draw (45:2) arc (45:135:2);
	\draw[dashed] (45:2) to (-45:2);
	\draw[dashed] (135:2) to (-135:2);

	\node[dot] (a) at (-72:2) {};
	\node[dot] (b) at (-108:2) {};
	\node[below] at (a) {$\ell_{i}$};
	\node[below] at (b) {$\ell_{i+1}$};
	\node[invisible] (x) at (0,-.5) {};
	\node[invisible] (y) at (-.4,-.9) {};
	\node[invisible] (z) at (.4,-.9) {};
	\node[invisible] (t) at (0,-1.3) {};
	\node[invisible] (u) at (-.7,-.3) {};

	\draw[faded oriented] (b) to [out=112.5,in=225] (y);
	\draw[oriented,draw=black!25] (y) to [out=45,in=225] (x);
	\draw[oriented,draw=black!25] (x) to [out=45,in=180] (1.41,.5);
	
	\draw[oriented,draw=black!25] (-1.41,.5) to [out=0,in=90] (u);
	\draw[draw=black!25] (u) to [out=270,in=135] (y);
	\draw[oriented,draw=black!25] (y) to [out=-45,in=135] (t);
	\draw[oriented,draw=black!25] (t) to [out=-45,in=157.5] (a);
	
	\draw[oriented,draw=black!25] (a) to [out=67.5,in=-45] (z);
	\draw[oriented,draw=black!25] (z) to [out=135,in=-45] (x);
	\draw[oriented,draw=black!25] (x) to [out=135,in=0] (u);
	\draw[oriented,draw=black!25] (u) to [out=180,in=0] (-1.41,-.5);
	
	\draw[oriented,draw=black!25] (1.41,-.5) to [out=180,in=45] (z);
	\draw[oriented,draw=black!25] (z) to [out=225,in=45] (t);
	\draw[oriented,draw=black!25] (t) to [out=225,in=22.5] (b);
	
	\node[blue] (2) at (-.1,.75) {$y$};
	\node[blue dot] (1) at (0,-.9) {};
	\node[blue,left] at (1) {$x$};

	\node at (0,-3.25) {Region $y$ on boundary};
\end{scope}
\begin{scope}[scale=1]
	\path[use as bounding box] (-2.5,-1.8) rectangle (2.5,2);
	\draw[dashed] (2,0) arc (0:180:2);
	\begin{scope}
		\clip (0,0) circle (2);
		\draw (0,2) circle (2.82);
	\end{scope}
	\node[dot] (a) at (1.08, -.605) {};
	\node[dot] (b) at (-1.08, -.605) {};
	\node[below] at (a) {$\ell_{i}$};
	\node[below] at (b) {$\ell_{i+1}$};
	\node[invisible] (x) at (0,.9) {};
	\node[invisible] (y) at (-.4,.5) {};
	\node[invisible] (z) at (.4,.5) {};
	\node[invisible] (t) at (0,.1) {};
	\node[invisible] (u) at (-.7,1.1) {};

	\draw[faded oriented] (b) to [out=112.5,in=225] (y);
	\draw[oriented,draw=black!25] (y) to [out=45,in=225] (x);
	\draw[oriented,draw=black!25] (x) to [out=45,in=-120] (60:2);
	
	\draw[oriented,draw=black!25] (120:2) to [out=-60,in=90] (u);
	\draw[oriented,draw=black!25] (u) to [out=270,in=135] (y);
	\draw[oriented,draw=black!25] (y) to [out=-45,in=135] (t);
	\draw[oriented,draw=black!25] (t) to [out=-45,in=157.5] (a);
	
	\draw[oriented,draw=black!25] (a) to [out=67.5,in=-45] (z);
	\draw[oriented,draw=black!25] (z) to [out=135,in=-45] (x);
	\draw[oriented,draw=black!25] (x) to [out=135,in=0] (u);
	\draw[oriented,draw=black!25] (u) to [out=180,in=-30] (150:2);
	
	\draw[oriented,draw=black!25] (30:2) to [out=-150,in=45] (z);
	\draw[oriented,draw=black!25] (z) to [out=225,in=45] (t);
	\draw[oriented,draw=black!25] (t) to [out=225,in=22.5] (b);
		
	\node[blue dot] (2) at (-.1,1.45) {};
	\node[blue dot] (1) at (0,.5) {};
	\node[blue,left] at (2) {$y$};
	\node[blue,left] at (1) {$x$};
	\draw[-angle 90, thick, blue] (1) to node[right,blue] {$\alpha$} (2);
	\begin{scope}
	\clip (0,0) circle (2);
	\draw[thick, blue!50] (2) to (-1,2.5);
	\draw[thick, blue!50,angle 90-] (2) to (0,2.5);
	\draw[thick, blue!50] (2) to (1,2.5);
	\end{scope}

	\node at (0,-1.75) {Region $y$ in interior};
\end{scope}
\end{tikzpicture}
\caption{Two scenarios for $s_ivs_i$ (quiver drawn in blue)}
\label{fig: twocases}
\end{figure}
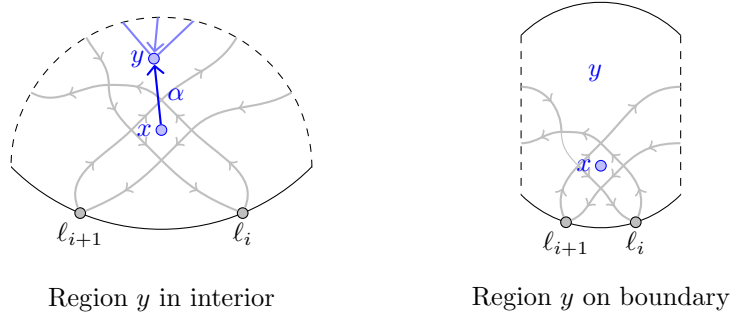

If $y$ is a mutable vertex, then we are in the first case above. If $y$ is a frozen vertex, we are in the second case. In either case, the Lemma is proved.
\end{proof}

\begin{corollary} \label{sideways}
With notation and hypotheses as in Lemma~\ref{covers}, $\cA(s_i v)$ is Louise if and only is $\cA(v s_i)$ is.
\end{corollary}

\begin{proof}
$\cA(s_i v)$ and $\cA(v s_i)$ have the same exchange type; the Louise property depends only on exchange type.
\end{proof}

\begin{corollary} \label{downwards}
With notation and hypotheses as in Lemma~\ref{covers}, if $\cA(s_i v)$ and $\cA(s_i v s_i)$ are Louise, so is $\cA(v)$.
\end{corollary}

\begin{proof}
Suppose that case (1) of Lemma~\ref{covers} applies. Then $\cA(v)$ can be represented by a quiver with two vertices $x$ and $y$, so that $x$ is a source or sink joined to $y$, and localizing at $x$, at $y$ or at both, produces Louise quivers. On the other hand, if case (2) applies, then the quiver of $\cA(v)$ is the disjoint union of a Louise quiver and a single vertex.
\end{proof}

The power of Corollary~\ref{sideways} comes from the fact that, although cluster algebras of the same exchange type have the same set of mutable quivers, there can be mutable quivers coming from Postnikov diagrams for $\cA(s_i v)$ where the corresponding mutable quiver for $\cA(v s_i)$ does not come from a Postnikov diagram. 
Thus, Corollary~\ref{sideways} lets us use Postnikov's technology to describe more mutable quivers.

\begin{example}
We give an example of the point discussed in the previous paragraph.
Consider the bounded affine permutation $w$ of type $(3,9)$ defined below.
\[w=\left\{\begin{array}{ccccccccc}
1 & 2 & 3 & 4 & 5 & 6 & 7 & 8 & 9 \\
\downarrow & \downarrow & \downarrow & \downarrow & \downarrow & \downarrow & \downarrow & \downarrow & \downarrow \\
3 & 8 & 7 & 6 & 2 & 10 & 9 & 14 & 13 \end{array}\right\}
\]
A Postnikov diagram for this permutation is given in Figure~\ref{fig: rigid}.  

\begin{figure}[h!t]
\begin{tikzpicture}
\begin{scope}[scale=.65]
		\draw (0,0) circle (4);
		\node[invisible] (1) at (130:4) {};
		\node[invisible] (2) at (90:4) {};
		\node[invisible] (3) at (50:4) {};
		\node[invisible] (4) at (10:4) {};
		\node[invisible] (5) at (330:4) {};
		\node[invisible] (6) at (290:4) {};
		\node[invisible] (7) at (250:4) {};
		\node[invisible] (8) at (210:4) {};
		\node[invisible] (9) at (170:4) {};
				
		\node[above left] at (1) {$\ell_1$};
		\node[above] at (2) {$\ell_2$};
		\node[above right] at (3) {$\ell_3$};
		\node[right] at (4) {$\ell_4$};
		\node[right] at (5) {$\ell_5$};
		\node[below right] at (6) {$\ell_6$};
		\node[below left] at (7) {$\ell_7$};
		\node[left] at (8) {$\ell_8$};
		\node[left] at (9) {$\ell_9$};

		\node[invisible] (1v) at (120:3) {};
		\node[invisible] (2v) at (90:3) {};
		\node[invisible] (3v) at (60:3) {};
		\node[invisible] (4v) at (0:3) {};
		\node[invisible] (5v) at (330:3) {};
		\node[invisible] (6v) at (300:3) {};
		\node[invisible] (7v) at (240:3) {};
		\node[invisible] (8v) at (210:3) {};
		\node[invisible] (9v) at (180:3) {};
		
		\node[invisible] (a) at (30:1.6) {};
		\node[invisible] (b) at (270:1.6) {};
		\node[invisible] (c) at (150:1.6) {};
		\node[invisible] (d) at (0,0) {};

		\PostnikovBlackBoundaryThree{1v}{1}{2v}{c};
		\PostnikovWhiteBoundaryThree{2v}{2}{3v}{1v};
		\PostnikovBlackBoundaryThree{3v}{3}{a}{2v};
		\PostnikovBlackBoundaryThree{4v}{4}{5v}{a};
		\PostnikovWhiteBoundaryThree{5v}{5}{6v}{4v};
		\PostnikovBlackBoundaryThree{6v}{6}{b}{5v};
		\PostnikovBlackBoundaryThree{7v}{7}{8v}{b};
		\PostnikovWhiteBoundaryThree{8v}{8}{9v}{7v};
		\PostnikovBlackBoundaryThree{9v}{9}{c}{8v};
		\PostnikovWhite{a}{3v}{4v}{d};
		\PostnikovWhite{b}{6v}{7v}{d};
		\PostnikovWhite{c}{9v}{1v}{d};
		\PostnikovBlack{d}{a}{b}{c};

\end{scope}
\begin{scope}[xshift=2.5in,scale=.65]
		\draw (0,0) circle (4);
		\node[invisible] (1) at (130:4) {};
		\node[invisible] (2) at (90:4) {};
		\node[invisible] (3) at (50:4) {};
		\node[invisible] (4) at (10:4) {};
		\node[invisible] (5) at (330:4) {};
		\node[invisible] (6) at (290:4) {};
		\node[invisible] (7) at (250:4) {};
		\node[invisible] (8) at (210:4) {};
		\node[invisible] (9) at (170:4) {};
				
		\node[above left] at (1) {$\ell_1$};
		\node[above] at (2) {$\ell_2$};
		\node[above right] at (3) {$\ell_3$};
		\node[right] at (4) {$\ell_4$};
		\node[right] at (5) {$\ell_5$};
		\node[below right] at (6) {$\ell_6$};
		\node[below left] at (7) {$\ell_7$};
		\node[left] at (8) {$\ell_8$};
		\node[left] at (9) {$\ell_9$};

		\node[invisible] (1v) at (120:3) {};
		\node[invisible] (2v) at (90:3) {};
		\node[invisible] (3v) at (60:3) {};
		\node[invisible] (4v) at (0:3) {};
		\node[invisible] (5v) at (330:3) {};
		\node[invisible] (6v) at (300:3) {};
		\node[invisible] (7v) at (240:3) {};
		\node[invisible] (8v) at (210:3) {};
		\node[invisible] (9v) at (180:3) {};
		
		\node[invisible] (a) at (30:1.6) {};
		\node[invisible] (b) at (270:1.6) {};
		\node[invisible] (c) at (150:1.6) {};
		\node[invisible] (d) at (0,0) {};

		\PostnikovBlackBoundaryThree[faded oriented]{1v}{1}{2v}{c};
		\PostnikovWhiteBoundaryThree[faded oriented]{2v}{2}{3v}{1v};
		\PostnikovBlackBoundaryThree[faded oriented]{3v}{3}{a}{2v};
		\PostnikovBlackBoundaryThree[faded oriented]{4v}{4}{5v}{a};
		\PostnikovWhiteBoundaryThree[faded oriented]{5v}{5}{6v}{4v};
		\PostnikovBlackBoundaryThree[faded oriented]{6v}{6}{b}{5v};
		\PostnikovBlackBoundaryThree[faded oriented]{7v}{7}{8v}{b};
		\PostnikovWhiteBoundaryThree[faded oriented]{8v}{8}{9v}{7v};
		\PostnikovBlackBoundaryThree[faded oriented]{9v}{9}{c}{8v};
		\PostnikovWhite[faded oriented]{a}{3v}{4v}{d};
		\PostnikovWhite[faded oriented]{b}{6v}{7v}{d};
		\PostnikovWhite[faded oriented]{c}{9v}{1v}{d};
		\PostnikovBlack[faded oriented]{d}{a}{b}{c};
		
		\node[mutable,blue, fill=blue!10] (M1) at (90:1.6) {};
		\node[mutable,blue, fill=blue!10] (M2) at (330:1.6) {};
		\node[mutable,blue, fill=blue!10] (M3) at (210:1.6) {};	

		\draw[blue,-angle 90] (M1) to (M2);
		\draw[blue,-angle 90] (M2) to (M3);
		\draw[blue,-angle 90] (M3) to (M1);
		
\end{scope}
\end{tikzpicture}
\caption{A Postnikov diagram of type $(3,9)$ and its mutable quiver}
\label{fig: rigid}
\end{figure}
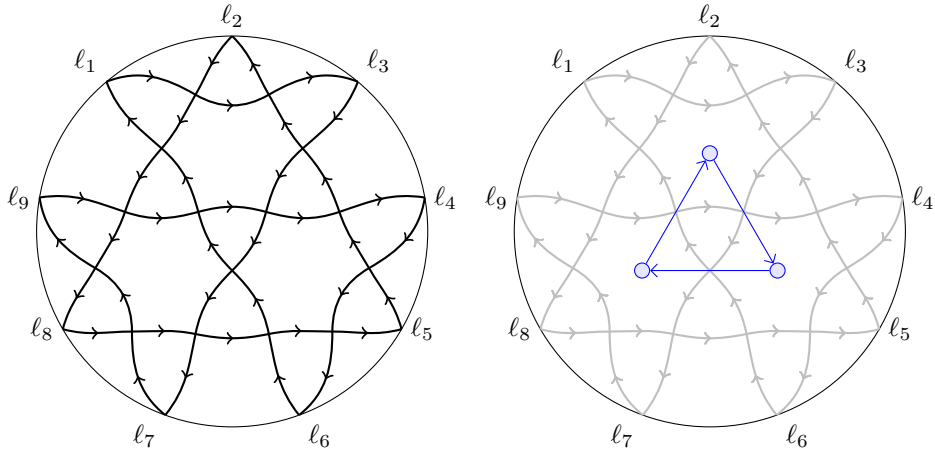


This is the only Postinkov diagram for this permutation, up to trivial modifications which don't change the underlying seed.
In particular, because the alternating regions associated with mutable variables are all hexagons, there are no square moves available.
Consequently, the associated cluster algebra only has one seed which may be described by a Postnikov diagram, and the mutable quiver of this seed has no sources or sinks.   Nevertheless, this cluster algebra is acyclic\footnote{In fact, it is a finite-type cluster algebra with $9$ cluster variables and $14$ seeds, $12$ of which are acyclic seeds.}.  

This is an example in which there are not `enough' seeds coming from Postnikov diagrams to prove local acyclicity.  Hence, we must use Lemma \ref{covers} to alter the permutation $w$, so as to change the set of seeds described by Postnikov diagrams without changing the exchange type.  In particular, $v=ws_8$ satisfies the hypothesis of Lemma \ref{covers} for $i=8$, and so Corollary~\ref{sideways} tells us we may consider the bounded affine permutation $s_8v = s_8ws_8$ instead.  Figure~\ref{fig: rigid2} shows the Postnikov diagram whose mutable quiver matches that of Figure~\ref{fig: rigid}.  Note that one of the alternating regions corresponding to a mutable variable is now a quadrilateral, so we may perform a square move at that region. Doing so obtains a Postnikov diagram for an acyclic seed of $s_8 w s_8$, and hence shows that there are acyclic seeds for $w$.

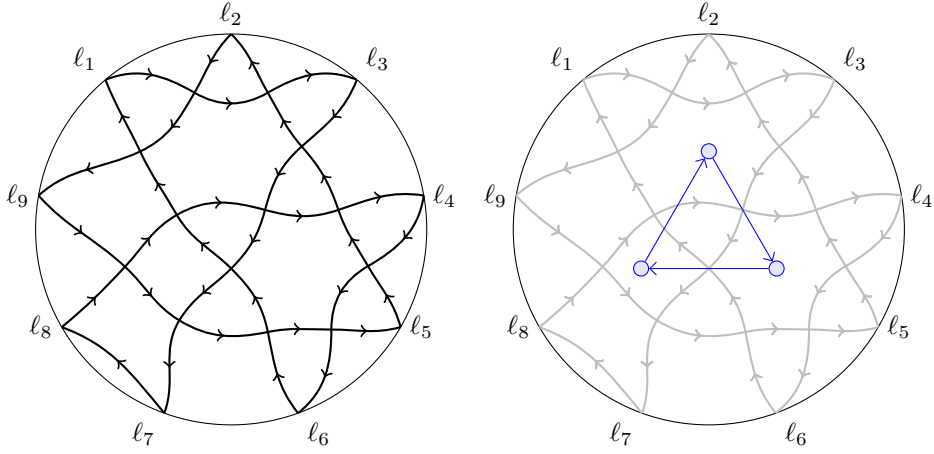
\begin{figure}[h!t]
\begin{tikzpicture}
\begin{scope}[scale=.65]
		\draw (0,0) circle (4);
		\node[invisible] (1) at (130:4) {};
		\node[invisible] (2) at (90:4) {};
		\node[invisible] (3) at (50:4) {};
		\node[invisible] (4) at (10:4) {};
		\node[invisible] (5) at (330:4) {};
		\node[invisible] (6) at (290:4) {};
		\node[invisible] (7) at (250:4) {};
		\node[invisible] (8) at (210:4) {};
		\node[invisible] (9) at (170:4) {};
				
		\node[above left] at (1) {$\ell_1$};
		\node[above] at (2) {$\ell_2$};
		\node[above right] at (3) {$\ell_3$};
		\node[right] at (4) {$\ell_4$};
		\node[right] at (5) {$\ell_5$};
		\node[below right] at (6) {$\ell_6$};
		\node[below left] at (7) {$\ell_7$};
		\node[left] at (8) {$\ell_8$};
		\node[left] at (9) {$\ell_9$};

		\node[invisible] (1v) at (120:3) {};
		\node[invisible] (2v) at (90:3) {};
		\node[invisible] (3v) at (60:3) {};
		\node[invisible] (4v) at (0:3) {};
		\node[invisible] (5v) at (330:3) {};
		\node[invisible] (6v) at (300:3) {};
		
		\node[invisible] (a) at (30:1.6) {};
		\node[invisible] (b) at (270:1.6) {};
		\node[invisible] (d) at (0,0) {};

		\node[invisible] (78) at (225:3) {};
		\node[invisible] (9c) at (165:2.3) {};
		
		\PostnikovBlackBoundaryThree{1v}{1}{2v}{9c};
		\PostnikovWhiteBoundaryThree{2v}{2}{3v}{1v};
		\PostnikovBlackBoundaryThree{3v}{3}{a}{2v};
		\PostnikovBlackBoundaryThree{4v}{4}{5v}{a};
		\PostnikovWhiteBoundaryThree{5v}{5}{6v}{4v};
		\PostnikovBlackBoundaryThree{6v}{6}{b}{5v};
		
		\PostnikovWhiteBIII{9c}{9}{1v}{d}{78};
		\PostnikovBlackBBII{78}{7}{8}{9c}{b};
		
		\PostnikovWhite{a}{3v}{4v}{d};
		\PostnikovWhite{b}{6v}{78}{d};
		\PostnikovBlack{d}{a}{b}{9c};

\end{scope}
\begin{scope}[xshift=2.5in,scale=.65]
		\draw (0,0) circle (4);
		\node[invisible] (1) at (130:4) {};
		\node[invisible] (2) at (90:4) {};
		\node[invisible] (3) at (50:4) {};
		\node[invisible] (4) at (10:4) {};
		\node[invisible] (5) at (330:4) {};
		\node[invisible] (6) at (290:4) {};
		\node[invisible] (7) at (250:4) {};
		\node[invisible] (8) at (210:4) {};
		\node[invisible] (9) at (170:4) {};
				
		\node[above left] at (1) {$\ell_1$};
		\node[above] at (2) {$\ell_2$};
		\node[above right] at (3) {$\ell_3$};
		\node[right] at (4) {$\ell_4$};
		\node[right] at (5) {$\ell_5$};
		\node[below right] at (6) {$\ell_6$};
		\node[below left] at (7) {$\ell_7$};
		\node[left] at (8) {$\ell_8$};
		\node[left] at (9) {$\ell_9$};

		\node[invisible] (1v) at (120:3) {};
		\node[invisible] (2v) at (90:3) {};
		\node[invisible] (3v) at (60:3) {};
		\node[invisible] (4v) at (0:3) {};
		\node[invisible] (5v) at (330:3) {};
		\node[invisible] (6v) at (300:3) {};
		
		\node[invisible] (a) at (30:1.6) {};
		\node[invisible] (b) at (270:1.6) {};
		\node[invisible] (d) at (0,0) {};

		\node[invisible] (78) at (225:3) {};
		\node[invisible] (9c) at (165:2.3) {};
		
		\PostnikovBlackBoundaryThree[faded oriented]{1v}{1}{2v}{9c};
		\PostnikovWhiteBoundaryThree[faded oriented]{2v}{2}{3v}{1v};
		\PostnikovBlackBoundaryThree[faded oriented]{3v}{3}{a}{2v};
		\PostnikovBlackBoundaryThree[faded oriented]{4v}{4}{5v}{a};
		\PostnikovWhiteBoundaryThree[faded oriented]{5v}{5}{6v}{4v};
		\PostnikovBlackBoundaryThree[faded oriented]{6v}{6}{b}{5v};
		
		\PostnikovWhiteBIII[faded oriented]{9c}{9}{1v}{d}{78};
		\PostnikovBlackBBII[faded oriented]{78}{7}{8}{9c}{b};
		
		\PostnikovWhite[faded oriented]{a}{3v}{4v}{d};
		\PostnikovWhite[faded oriented]{b}{6v}{78}{d};
		\PostnikovBlack[faded oriented]{d}{a}{b}{9c};
		
		\node[mutable,blue, fill=blue!10] (M1) at (90:1.6) {};
		\node[mutable,blue, fill=blue!10] (M2) at (330:1.6) {};
		\node[mutable,blue, fill=blue!10] (M3) at (210:1.6) {};	

		\draw[blue,-angle 90] (M1) to (M2);
		\draw[blue,-angle 90] (M2) to (M3);
		\draw[blue,-angle 90] (M3) to (M1);
		
\end{scope}
\end{tikzpicture}
\caption{A Postnikov diagram with a different permutation but the same mutable quiver as Figure \ref{fig: rigid}}
\label{fig: rigid2}
\end{figure}
\end{example}

We introduce one more combinatorial tool: For a bounded affine permutation $w$, we define the \newword{smallest throw of $w$} to be the minimum value of $w(i) - i$.
We are now ready to prove Theorem~\ref{MainTheorem}.

\begin{proof}[Proof of Theorem~\ref{MainTheorem}] For any bounded affine permutation $w$, the cluster algebra $\cA(w)$ is Louise.
Let $w$ be a bounded affine permutation of type $(k,n)$, length $\ell$ and shortest throw $t$.
We assume inductively that Theorem~\ref{MainTheorem} has already been proved:
\begin{itemize}
\item for all bounded affine permutations on fewer than $n$ points
\item for all bounded affine permutations on $n$ points of length $> \ell$
\item and for all bounded affine permutations on $n$ points of length $\ell$ with shortest throw $<t$. 
\end{itemize}
We can take the base cases to be the (unique) bounded affine permutations of types $(0,1)$ and $(1,1)$, which correspond to quivers with no mutable vertices.

Let $i$ be one of the indices for which $w(i) - i$ equals $t$. We break into cases:

\textbf{Case 1:} $t=0$ or $n$. In this case, by Lemma~\ref{fixed point}, there is a bounded affine permutation $w'$ on $n-1$ points so that $\cA(w)$ and $\cA(w')$ have the same exchange type; by induction, $\cA(w')$ is Louise.

\textbf{Case 2:} $t=1$ or $n-1$. In this case, by Lemma~\ref{short arc}, there is a bounded affine permutation $w'$ on $n$ points with length $\ell+1$ so that so that $\cA(w)$ and $\cA(w')$ have the same exchange type; by induction, $\cA(w')$ is Louise.

\textbf{Case 3:} $2 \leq t \leq n-2$. Since $t$ is the shortest throw, we have $w(i+1) \geq i+1 + t > i+t = w(i)$. We now break into two cases:

\textbf{Case 3a:} $w^{-1}(i+1) > w^{-1}(i)$. By Corollary~\ref{downwards} with $w=v$, if $\cA(s_i w)$ and $\cA(s_i w s_i)$ are Louise, then so is $\cA(w)$.
We have $\ell(s_i w) = \ell+1$ and $\ell(s_i w s_i) = \ell+2$, so we inductively know that $\cA(s_i w)$ and $\cA(s_i w s_i)$ are Louise.

\textbf{Case 3b:} $w^{-1}(i+1) < w^{-1}(i)$. We use Corollary~\ref{sideways} with $v=w s_i$. This says that $\cA(w)$ and $\cA(s_i w s_i)$ have the same exchange type.
We have $\ell(s_i w s_i)=\ell(w)$ in this case, and $(s_i w s_i)(i+1) = i+t$, so the shortest throw of $s_iws_i$ is less than $t$. So, again, we inductively know that $\cA(s_i w s_i)$ is Louise.
\end{proof}

\begin{remark}\label{rem: nondeg}
We have in fact proved a slightly stronger condition: $\cA(w)$ is Louise in such a manner that, whenever we make use of a cover $\Spec(\cA) = \Spec(\cA[x_s^{-1}]) \cup \Spec(\cA[x_t^{-1}])$, the vertex $s$ is a source in the mutable quiver.  

This has some additional consequences.  First, any quantum cluster algebra $\cA_q$ (in the sense of \cite{BZ05}) whose exchange type is the same as the exchange type of a Postnikov diagram will equal its own quantum upper cluster algebra, by \cite[Lemma 8.13]{MulSk}.  Second, the mutable quivers of Postnikov diagrams admit a unique non-degenerate potential up to weak equivalence.  This is a consequence of the observation that adding a source to a quiver does not change whether it has a unique non-degenerate potential up to weak equivalence.
\end{remark}

\bibliography{MyNewBib}
\bibliographystyle{amsalpha}

\end{document}